\documentclass[a4paper,12pt]{amsart} 

\usepackage{amssymb,amsmath}
\usepackage{pb-diagram}
\usepackage{pb-xy}
\usepackage[cmtip,arrow]{xy}

\newtheorem{thm}{Theorem}[section] 
\newtheorem*{thm*}{Theorem}
\newtheorem{prop}[thm]{Proposition}
\newtheorem{cor}[thm]{Corollary} 
\newtheorem{lem}[thm]{Lemma}
  
\theoremstyle{definition} 
 
\newtheorem{rem}[thm]{Remark} 
\newtheorem{exa}[thm]{Example}

\newtheorem*{Ackn}{Acknowledgment}

\numberwithin{equation}{section}

\newcommand{\skal}[2]{\langle #1,#2\rangle}
\newcommand{\alg}[1]{\mathfrak{#1}}
\newcommand{\algg}{\mathfrak{g}}

\begin{document}

\title[An integral formula for Riemannian $G$--structures]{An integral formula for Riemannian $G$--structures with applications to almost Hermitian and almost contact structures}
\author{Kamil Niedzia\l omski}
\date{}

 
\address{
Department of Mathematics and Computer Science \endgraf
University of \L\'{o}d\'{z} \endgraf
ul. Banacha 22, 90-238 \L\'{o}d\'{z} \endgraf
Poland
}
\email{kamiln@math.uni.lodz.pl}

\begin{abstract}
For a Riemannian $G$--structure, we compute the divergence of the vector field induced by the intrinsic torsion. Applying the Stokes theorem, we obtain the integral formula on a closed oriented Riemannian manifold, which we interpret in certain cases. We focus on almost Hermitian and almost contact metric structures.
\end{abstract}

\maketitle

\section{Introduction}

Equipping an $n$--dimensional manifold $M$ with a Riemannian metric $g$ is equivalent to the reduction of a frame bundle $L(M)$ to the orthogonal frame bundle $O(M)$, i.e. to action of a structure group $O(n)$. Assuming moreover that $M$ is oriented we can consider the bundle $SO(M)$ of oriented orthonormal frames. Existence of additional geometric structure can be considered as a reduction of a structure group $SO(n)$ to a certain subgroup $G$. For example, almost Hermitian structure gives $U(\frac{n}{2})$--structure, almost contact metric structure is just a $U(\frac{n-1}{2})\times 1$--structure, etc.

If $\nabla$ is the Levi--Civita connection of $(M,g)$ we may measure the defect of $\nabla$ to be a $G$--connection. This leads to the notion of an intrinsic torsion. If this $(1,2)$--tensor vanishes (in such case we say that a $G$--structure is integrable) then $\nabla$ is a $G$--connection, which implies that the holonomy group is contained in $G$. We may classify non--integrable geometries by finding the decomposition of the space of all possible intrinsic torsions into irreducible $G$--modules. This approach was initiated by Gray and Hervella for $U(\frac{n}{2})$--structures \cite{GH} and later considered for other structures by many authors \cite{CG,FMC,AS,FG,CS}. Each, so--called, Gray--Hervella class, gives some restrictions on the curvature.

One possible approach to curvature restrictions on compact $G$--structures can be achieved by obtaining integral formulas relating considered objects. This has been firstly done, in a general case, by Bor and Hern\'{a}ndez Lamoneda \cite{BL}. They uses Bochner--type formula for forms being stabilizers of each considered subgroup in $SO(n)$. They obtained integral formulas for $G=U(\frac{n}{2}),SU(\frac{n}{2}), G_2$ and ${\rm Spin}_7$ and continued this approach for $Sp(n)Sp(1)$ in \cite{BL2}. The case $G=U(\frac{n-1}{2})\times 1$ has been studied later in \cite{GMC2} by other authors.

In this article, we show how mentioned formulas can be obtained in a different way. The nice feature of our approach is that the main integral formula 
\begin{equation*}
\frac{1}{2}\int_M s_{\alg{g}^{\bot}}-s^{\rm alt}_{\alg{g}^{\bot}}\,{\rm vol}_M=
\int_M |\chi|^2+|\xi^{\rm alt}|^2-|\xi^{\rm sym}|^2\,{\rm vol}_M.
\end{equation*} 
is valid for any $G$--structure on closed $M$ for compact $G\subset SO(n)$. Let us roughly describe the approach and all used objects in this formula. We consider, so--called, characteristic vector field $\chi=\sum_i \xi_{e_i}e_i$ induced by the intrinsic torsion $\xi$ and calculate its divergence. $\xi^{\rm alt}$ and $\xi^{\rm sym}$ denote the skew--symmetric and symmetric components of $\xi$, $\xi^{\rm alt}_XY=\frac{1}{2}(\xi_XY-\xi_YX)$, $\xi^{\rm sym}_XY=\frac{1}{2}(\xi_XY+\xi_YX)$, whereas, $s_{\alg{g}^{\bot}}$ and $s^{\rm alt}_{\alg{g}^{\bot}}$ are, in a sense, $\alg{g}^{\bot}$ components of a scalar curvature (see the following sections for more details). For some Gray--Hervella classes the characteristic vector field vanishes, and then we get point--wise formula relating an intrinsic torsion with a curvature.

We concentrate on almost Hermitian and almost contact metric structures. In the way described above we recover many well known relations. Let us state some of the consequences of the main itegral formula (the objects used in these statements will be defined in appropriate sections): 
\begin{enumerate}
\item Assume $(M,g,J)$ is closed Hermitian manifold of Gray--Hervella type $\mathcal{W}_4$ such that $s=s^{\ast}$, where $s$ is a scalar curvature and $s^{\ast}$ is a $\ast$--scalar curvature. Then $M$ is K\"ahler (compare \cite{BL}).
\item On a closed $SU(n)$--structure of type $\mathcal{W}_1\oplus\mathcal{W}_5$ we have $\int_M s=5\int_M s^{\ast}$.
\item Let $(M,g,\varphi,\eta,\zeta)$ be an almost contact metric structure with the intrinsic torsion $\xi\in\mathcal{D}_2$. Then
\begin{equation*}
{\rm div}(\nabla_{\zeta}\zeta)=\frac{1}{2}s^{\rm alt}_{\alg{u}(n)^{\bot}}+\frac{1}{2}{\rm Ric}(\zeta,\zeta)-\frac{1}{4}(s-s^{\ast}),
\end{equation*}
where $s$ is a scalar curvature and $s^{\ast}$ is an associated $\ast$--scalar curvature.
\end{enumerate} 

In the end, we consider some examples focusing on (reductive) homogeneous spaces. We show, which is an immediate consequence of the formula for the Levi--Civita connection, that in these examples the characteristic vector field vanishes. Hence, the main divergence formula is point--wise.  

\begin{Ackn}
I wish to thank Ilka Agricola for indication of references \cite{BFGK} and \cite{IA0} and helpful conversations.

\noindent
The author is partially supported by the National Science Center, Poland -- Grant Miniatura 2017/01/X/ST1/01724
\end{Ackn}

\section{Intrinsic torsion}

Let $(M,g)$ be an oriented Riemannian manifold. Denote by $SO(M)$ the bundle of oriented frames over $M$. Let $\nabla$ be the Levi-Civita connection of $g$ and let $\omega$ be the induced connection form. Let $G\subset SO(n)$, where $n=\dim M$, be a closed subgroup. Then on the level of Lie algebras we have the following decomposition
\begin{equation*}
\alg{so}(n)=\algg\oplus\algg^{\bot},\quad {\rm ad}(G)\algg^{\bot}\subset\algg^{\bot},
\end{equation*} 
where the orthogonal complement is taken with respect to the Killing form. Hence $\omega$ decomposes as
\begin{equation*}
\omega=\omega_{\algg}\oplus\omega_{\algg^{\bot}},
\end{equation*}
where $\omega_{\algg}$ is a connection form in the $G$--reduction $P\subset SO(M)$, if such exists, and therefore defines a Riemannian connection $\nabla^G$ on $M$. The difference
\begin{equation*}
\xi_XY=\nabla^G_XY-\nabla_XY,\quad X,Y\in TM,
\end{equation*}
defines a $(1,2)$--tensor called {\it the intrinsic torsion} of a $G$--structure. $\xi$ satisfies some skew--symmetry conditions by the fact that $\xi_X\in\algg^{\bot}(TM)\subset \alg{so}(TM)$ where $\algg^{\bot}(TM)$ is the associated bundle of the form $P\times_{\rm ad(G)}\algg^{\bot}$. In particular,
\begin{equation*}
g(\xi_XY,Z)=-g(Y,\xi_XZ),\quad X,Y,Z\in TM.
\end{equation*} 

By a definition, the intrinsic torsion measures the defect of the Levi--Civita connection to be a $G$--connection. In particular, if $\xi$ vanishes, then the holonomy of $\nabla$ is contained in $G$. The study of the intrinsic torsion and its decomposition into irreducible summands was initiated by Gray and Hervella in the case of $G=U(\frac{n}{2})$ \cite{GH}. Since then, other possible cases, mainly coming from the Berger classification of non--symmetric irreducible holonomy groups, has been considered (see, for example, \cite{CG,FMC,AS,FG,CS}).

\section{An integral formula}

Let $(M,g)$ be an oriented Riemannian manifold with the Levi-Civita connection $\nabla$. Assume $M$ is a $G$--structure, with $G\subset SO(n)$ and let $\xi$ be the associated intrinsic torsion. Define a vector field $\chi=\chi^G$ by
\begin{equation}\label{eQ:characteristicvf}
\chi=\sum_i \xi_{e_i}e_i,
\end{equation}
where $(e_i)$ is any orthonormal basis. We call $\chi$ {\it the characteristic vector field} of a $G$--structure $M$. Notice that if $\xi$ is skew--symmetric with respect to $X$ and $Y$ then $\chi$ vanishes. This is the case, for example, for nearly K\"ahler manifolds (see the following sections). Additionally,
\begin{equation}\label{eq:divergences}
g(\chi,X)=-\sum_i g(e_i,\xi_{e_i}X)={\rm div}X-{\rm div}^G X.
\end{equation}
Thus, vanishing of the characteristic vector field is equivalent to the fact that divergences with respect to $\nabla$ and $\nabla^G$ coincide. Moreover, put
\begin{equation}\label{eq:xialtsym}
\xi^{\rm alt}_XY=\frac{1}{2}(\xi_XY-\xi_YX)\quad\textrm{and}\quad \xi^{\rm sym}_XY=\frac{1}{2}(\xi_XY+\xi_YX). 
\end{equation}

In this section we compute the divergence of $\chi$ with respect to $\nabla$. First, let us recall well--known curvature identities involving the intrinsic torsion \cite{GMC2}:
\begin{equation}\label{eq:curvatureidentities}
\begin{split}
R(X,Y)_{\alg{g}} &=R^G(X,Y)+[\xi_X,\xi_Y]_{\alg{g}},\\
R( X,Y)_{\alg{g}^{\bot}} &=-(\nabla_X\xi)_Y+(\nabla_Y\xi)_X-2[\xi_X,\xi_Y]+[\xi_X,\xi_Y]_{\alg{g}^{\bot}},
\end{split}
\end{equation}
where $R$ and $R^G$ are the curvature tensors of $\nabla$ and $\nabla^G$, respectively. We use the following convention for the curvature $R(X,Y)=[\nabla_X,\nabla_Y]-\nabla_{[X,Y]}$. Thus
\begin{equation}\label{eq:curvatureidentity}
R^G(X,Y)=R(X,Y)+(\nabla_X\xi)_Y-(\nabla_Y\xi)_X+[\xi_X,\xi_Y].
\end{equation}
Denote by $s$ and $s^G$ the scalar curvatures of $R$ and $R^G$, respectively.

\begin{prop}\label{prop:divfor1}
On an oriented $G$--structure $M$ we have
\begin{equation}\label{eq:divformain2}
2{\rm div}\chi=s^G-s+|\chi|^2+|\xi^{\rm alt}|^2-|\xi^{\rm sym}|^2.
\end{equation} 
\end{prop}

\begin{proof}
By \eqref{eq:curvatureidentity} we have
\begin{equation}\label{eq:divfore1}
s^G=s+\sum_{i,j}g((\nabla_{e_i}\xi)_{e_j}e_j,e_i)-\sum_{i,j}g((\nabla_{e_j}\xi)_{e_i}e_j,e_i)+\sum_{i,j}g([\xi_{e_i},\xi_{e_j}]e_j,e_i).
\end{equation}
Notice that $(\nabla_X\xi)_Y$ is skew--symmetric, since $\xi_X$ is skew--symmetric,
\begin{align*}
g((\nabla_X\xi)_YZ,W) &=g(\nabla_X\xi_YZ,W)-g(\xi_{\nabla_XY}Z,W)-g(\xi_Y\nabla_XZ,W)\\
&=-X(Z,\xi_YW)-g(\xi_YZ,\nabla_XW)+g(Z,\xi_{\nabla_XY}W)+(\nabla_XZ,\xi_YW)\\
&=-g(Z,\nabla_X\xi_YW)+g(Z,\xi_Y\nabla_XW)+g(Z,\xi_{\nabla_XY}W)\\
&=-g(Z,(\nabla_X\xi)_YW).
\end{align*}
Thus the first and second sum on the right hand side of \eqref{eq:divfore1} are opposite. Moreover,
\begin{equation}\label{eq:divfore2}
{\rm div}\chi=\sum_{i,j}g(\nabla_{e_i}\xi_{e_j}e_j,e_i)=\sum_{i,j}g((\nabla_{e_i}\xi)_{e_j}e_j,e_i),
\end{equation}
since
\begin{align*}
\sum_{i,j}(g(\xi_{\nabla_{e_i}e_j}e_j,e_i)+g(\xi_{e_j}\nabla_{e_i}e_j,e_i)) &=\sum_{i,j,k}g(\nabla_{e_i}e_j,e_k)g(\xi_{e_k}e_j,e_i)+\sum_{i,j}g(\xi_{e_j}\nabla_{e_i}e_j,e_i)\\
&=-\sum_{i,j,k}g(e_j,\nabla_{e_i}e_k)g(\xi_{e_k}e_j,e_i)+\sum_{i,j}g(\xi_{e_j}\nabla_{e_i}e_j,e_i)\\
&=0.
\end{align*}
Let us compute the last term in \eqref{eq:divfore1},
\begin{equation}\label{eq:divfore3}
\begin{split}
\sum_{i,j}g([\xi_{e_i},\xi_{e_j}]e_j,e_i) &=\sum_{i,j}g(\xi_{e_i}\xi_{e_j}e_j,e_i)-g(\xi_{e_j}\xi_{e_i}e_j,e_i)\\
&=-|\chi|^2+\sum_{i,j}g(\xi_{e_j}e_i,\xi_{e_i}e_j)\\
&=-|\chi|^2+|\xi^{\rm sym}|^2-|\xi^{\rm alt}|^2.
\end{split}
\end{equation}
Substituting \eqref{eq:divfore2} and \eqref{eq:divfore3} into \eqref{eq:divfore1} we get \eqref{eq:divformain2}.
\end{proof}

We will improve above divergence formula a little bit, by getting rid of the component $s^G$ and replacing it by $\alg{g}^{\bot}$--component of $s$ and some additional term, which vanishes in some cases. Namely, denote by $s_{\alg{g}^{\bot}}^{\rm alt}$ the following quantity
\begin{equation*}
s_{\alg{g}^{\bot}}^{\rm alt}=\sum_{i,j}g([\xi_{e_i},\xi_{e_j}]_{\alg{g}^{\bot}}e_j,e_i).
\end{equation*}

\begin{prop}
On an oriented $G$--structure $M$ we have
\begin{equation}\label{eq:divformain}
{\rm div}\chi=\frac{1}{2}s_{\alg{g}^{\bot}}^{\rm alt}-\frac{1}{2}s_{\alg{g}^{\bot}}+|\chi|^2+|\xi^{\rm alt}|^2-|\xi^{\rm sym}|^2.
\end{equation} 
If $M$ is, additionally, closed, then the following integral formula holds
\begin{equation}\label{eq:intformain}
\frac{1}{2}\int_M s_{\alg{g}^{\bot}}-s^{\rm alt}_{\alg{g}^{\bot}}\,{\rm vol}=
\int_M |\chi|^2+|\xi^{\rm alt}|^2-|\xi^{\rm sym}|^2\,{\rm vol}.
\end{equation}
\end{prop}
\begin{proof}
By \eqref{eq:curvatureidentities} and \eqref{eq:divfore3} we have
\begin{align*}
s_{\alg{g}} &=s^G+\sum_{i,j}g([\xi_{e_i},\xi_{e_j}]e_j,e_i)-\sum_{i,j}g([\xi_{e_i},\xi_{e_j}]_{\alg{g}^{\bot}}e_j,e_i)\\
&=s^G-|\chi|^2-|\xi^{\rm alt}|^2+|\xi^{\rm sym}|^2-\sum_{i,j}g([\xi_{e_i},\xi_{e_j}]_{\alg{g}^{\bot}}e_j,e_i).
\end{align*}
Since $s=s_{\alg{g}}+s_{\alg{g}^{\bot}}$, \eqref{eq:divformain2} can be rewritten in the form \eqref{eq:divformain}.
\end{proof}

\begin{rem}\label{rem:invariants}
Notice that elements
\begin{equation*}
|\chi|^2,\quad |\xi^{\rm alt}|^2,\quad |\xi^{\rm sym}|^2
\end{equation*}
are quadratic invariants of the representation of $SO(n)$ in the space of $(1,2)$--tensors with the symmetries of the intrinsic torsion, i.e. the space $T^{\ast}M\otimes\alg{so}(TM)$ \cite{GH}. This implies that $|\xi|^2$ and $|\xi^{\rm alt}|^2-|\xi^{\rm sym}|^2$ are also quadratic invariants. Thus, for an irreducible submodule $\mathcal{U}$ of the representation $T^{\ast}M\otimes\alg{so}(TM)$, since the space of its quadratic invariants in one dimensional \cite{BGM}, then the number
\begin{equation*}
E_{\mathcal{U}}=|\chi^{\mathcal{U}}|^2+|\xi^{\mathcal{U},\rm alt}|^2-|\xi^{\mathcal{U},\rm sym}|^2 
\end{equation*} 
is a constant multiple of $|\xi^{\mathcal{U}}|^2$. Here $\xi^{\mathcal{U}}$ denotes the $\mathcal{U}$--component of $\xi$ with respect to decomposition into irreducible summands. This approach is also valid for any irreducible module $G$--module in the space of possible intrinsic torsions. This kind of approach, was used in \cite{BL} to get integral formulas for many $G$--structures.
\end{rem}

We have an immediate consequence of the formula \eqref{eq:divformain}.
\begin{cor}\label{cor:divfor}
Assume $M$ is an oriented $G$--structure, where $G=U(\frac{n}{2})$, $n$ even, or $G=SO(m)\times SO(n-m)$. If the characteristic vector field vanishes, then
\begin{equation*}
\frac{1}{2}s_{\alg{g}^{\bot}}=|\xi^{\rm alt}|^2-|\xi^{\rm sym}|^2.
\end{equation*}
In particular, if the intrinsic torsion is totally skew--symmetric, then
\begin{equation*}
s_{\alg{g}^{\bot}}=2|\xi|^2\geq 0
\end{equation*}
with the equality if and only if the $G$--structure $M$ is integrable (i.e. $\xi=0$).
\end{cor}
\begin{proof}
For the listed choices of $G$ we have $[\alg{g}^{\bot},\alg{g}^{\bot}]\subset \alg{g}$, thus $s^{\rm alt}_{\alg{g}^{\bot}}$ vanishes.
\end{proof}

The consequences of the integral formula will be presented in the following section for certain choices of $G$.

\section{Applications to certain Riemannian $G$--structures}

In this section we rewrite formulae \eqref{eq:divformain} and \eqref{eq:intformain} for certain $G$--structures. We also give some applications of these relations. We will show that obtained formulas are consistent with the Bochner type formulae obtained, using representation theory, in \cite{BL}.

\subsection{Almost product structures}

We show that the divergence and integral formulae obtained in the previous section agree with the Walczak formulas \cite{PW}. Since this integral formula has found many applications, we will only concentrate on deriving it from \eqref{eq:intformain} and state its one corollary, which will be needed later.

Let $(M,g)$ be an oriented Riemannian manifold, with two complementary orthogonal oriented distributions $\mathcal{D}$ and $\mathcal{D}^{\bot}$, i.e. $TM=\mathcal{D}\oplus\mathcal{D}^{\bot}$. Thus the bundle of oriented orthonormal frames $SO(M)$ has a reduction to a subgroup $SO(m)\times SO(n-m)\subset SO(n)$, where $m=\dim\mathcal{D}$. On the level of Lie algebras 
\begin{equation*}
\alg{so}(n)=(\alg{so}(m)\oplus\alg{so}(n-m))\oplus\alg{m},
\end{equation*}
where
\begin{equation*}
\alg{m}=(\alg{so}(m)\oplus\alg{so}(n-m))^{\bot}=\left\{
\left(\begin{array}{cc} 0 & A \\ -A^{\top} & 0\end{array}\right)\right\}
\end{equation*}
and $A$ is $m\times(n-m)$ matrix. Let $\nabla$ be the Levi--Civita connection of $g$. Since the orthogonal projection to $\alg{m}$ is just a restriction to non--diagonal blocks, it follows that the intrinsic torsion equals
\begin{equation*}
\xi_XY=-(\nabla_XY^{\top})^{\bot}-(\nabla_XY^{\bot})^{\top},
\end{equation*}
where $Y^{\top}$ and $Y^{\bot}$ denotes the components of $Y$ in $\mathcal{D}$ and $\mathcal{D}^{\bot}$. Notice that $\xi$ is made of shape operators and fundamental forms of distributions $\mathcal{D}$ and $\mathcal{D}^{\bot}$. Recall, that the second fundamental form, for example, of $\mathcal{D}$ is a $(1,2)$--symmetric tensor $B=B^{\mathcal{D}}$ of the form
\begin{equation*}
B(X,Y)=\frac{1}{2}(\nabla_XY+\nabla_YX)^{\bot},\quad X,Y\in \mathcal{D}.
\end{equation*}
Additionally, we will use integrability tensor $T=T^{\mathcal{D}}$ being just
\begin{equation*}
T(X,Y)=\frac{1}{2}[X,Y]^{\bot},\quad X,Y\in \mathcal{D}.
\end{equation*}
Notice that $B(X,Y)+T(X,Y)=(\nabla_XY)^{\bot}$, hence $B$ and $T$ are symmetrization and alternation of (a minus of) a part of the intrinsic torsion reduced to $\mathcal{D}$. 

For an orthonormal basis $(e_i)$ adapted to the decomposition $\mathcal{D}\oplus\mathcal{D}^{\bot}$, denote by $e_A$ elements of $(e_i)$ in $\mathcal{D}$ and by $e_{\alpha}$ elements of $(e_i)$ in $\mathcal{D}^{\bot}$. The characteristic vector field $\chi$ equals
\begin{equation}\label{eq:charvfdistributions}
\chi=-\sum_A (\nabla_{e_A}e_A)^{\bot}-\sum_{\alpha}(e_{\alpha}e_{\alpha})^{\top}=-H-H^{\bot},
\end{equation} 
where $H$ and $H^{\bot}$ are mean curvature vectors of $\mathcal{D}$ and $\mathcal{D}^{\bot}$ respectively.

We may now state and show that the Walczak formula \cite{PW} is an integral formula \eqref{eq:intformain} for $G=SO(m)\times SO(n-m)$.

\begin{prop}[\cite{PW}]\label{prop:Walczakformula}
On a closed Riemannian manifold equipped with a pair of complementary orthogonal and oriented distributions the following Walczak integral formula holds
\begin{equation}\label{eq:Walczakfor}
\int_M s_{\rm mix}=\int_M |H|^2+|H^{\bot}|^2+|T|^2+|T^{\bot}|^2-|B|^2-|B^{\bot}|^2,
\end{equation}
where $s_{\rm mix}$ is a mixed scalar curvature defined by
\begin{equation*}
s_{\rm mix}=\sum_{A,\alpha} g(R(e_A,e_{\alpha})e_{\alpha},e_A)
\end{equation*} 
\end{prop}
\begin{proof}
 Since
\begin{equation*}
\sum_{A,\alpha}|(\nabla_{e_{\alpha}}e_A)^{\bot}|^2
=\sum_{A,\alpha,\beta}g(\nabla_{e_{\alpha}}e_A,e_{\beta})^2
=\sum_{A,\alpha,\beta}g(\nabla_{e_{\alpha}}e_{\beta},e_A)^2
=\sum_{\alpha,\beta}|(\nabla_{e_{\alpha}}e_{\beta})^{\top}|^2
\end{equation*}
and analogously interchanging $e_{\alpha}$ with $e_A$, then
\begin{equation*}
|\xi^{\rm alt}|^2=|T|^2+|T^{\bot}|^2
+\frac{1}{4}\sum_{A,B}|(\nabla_{e_A}e_B)^{\bot}|^2+\frac{1}{4}\sum_{\alpha,\beta}|(\nabla_{e_{\alpha}}e_{\beta})^{\top}|^2
\end{equation*}
and
\begin{equation*}
|\xi^{\rm sym}|^2=|B|^2+|B^{\bot}|^2+\frac{1}{4}\sum_{\alpha,\beta}|(\nabla_{e_{\alpha}}e_{\beta})^{\top}|^2.
\end{equation*}
Moreover,
\begin{equation*}
s_{\alg{m}}=\sum_{i,j}g(R(e_i,e_j)_{\alg{m}}e_j,e_i)=2\sum_{A,\alpha}g(R(e_A,e_{\alpha})e_{\alpha},e_A)=2s_{\rm mix}.
\end{equation*}
Putting all these facts together \eqref{eq:divformain} implies Walczak divergence formula \cite{PW}
\begin{equation}\label{eq:Walczakdiv}
-{\rm div}(H+H^{\bot})=-s_{\rm mix}+|H|^2+|H^{\bot}|^2+|T|^2+|T^{\bot}|^2-|B|^2-|B^{\bot}|^2.
\end{equation}
Assuming $M$ is closed, the Walczak integral formula holds.
\end{proof} 

Formula \eqref{eq:Walczakfor} has found many applications. Let us only state one of its consequences for $\mathcal{D}$ of codimension $1$, since it will be used in one of forthcoming subsections. In this case, clearly, $T^{\bot}=0$ and $B^{\bot}=H^{\bot}$. Denoting the unit positively oriented vector field orthogonal to $\mathcal{D}$ by $\zeta$, we have $\chi=-({\rm div}\zeta)\zeta+\nabla_{\zeta}\zeta$. Moreover,
\begin{equation*}
s_{\rm mix}=\sum_A g(R(e_A,\zeta)\zeta,e_A)={\rm Ric}(\zeta,\zeta).
\end{equation*}
Therefore, \eqref{eq:Walczakdiv} and \eqref{eq:Walczakfor} can be rewritten in the following, well known way.
\begin{prop}\label{prop:Walczakcodim1}
On a Riemannian manifold with an orientable codimension one distribition $\mathcal{D}$, we have the following divergence formula
\begin{equation*}
{\rm div}(-({\rm div}\zeta)\zeta+\nabla_{\zeta}\zeta)=
{\rm Ric}(\zeta,\zeta)-({\rm div}\zeta)^2-|T|^2+|B|^2
\end{equation*}
and, assuming $M$ is closed, the following integral formula
\begin{equation}\label{eq:Walczakforcodim1}
\int {\rm Ric}(\zeta,\zeta)\,{\rm vol}_M=\int_M ({\rm div}\zeta)^2+|T|^2-|B|^2\,{\rm vol}_M.
\end{equation}
\end{prop}

\subsection{Almost Hermitian structures}

Assume $(M,g,J)$ is an oriented Riemannian manifold with an almost complex structure $J$, i.e., $J^2=-{\rm id}_{TM}$, which is Hermitian, i.e., $g(JX,JY)=g(X,Y)$ for $X,Y\in TM$. Then $(M,g,J)$ is of even dimension $2n$ and induces an $U(n)$--structure. On the level of Lie algebras, we have
\begin{equation*}
\alg{so}(2n)=\alg{u}(n)\oplus\alg{u}(n)^{\bot},
\end{equation*}
where
\begin{equation*}
\alg{u}(n)=\{ A\in\alg{so}(n)\mid AJ=JA\},\quad
\alg{u}(n)^{\bot}=\{A\in\alg{so}(n) \mid AJ=-JA\}.
\end{equation*}
In particular, $[\alg{u}(n)^{\bot},\alg{u}(n)^{\bot}]\subset\alg{u}(n)$, thus $s_{\alg{u}(n)^{\bot}}^{\rm alt}=0$. The orthogonal projection from $\alg{so}(n)$ to $\alg{u}(n)^{\bot}$ equals $A\mapsto \frac{1}{2}\left(A+JAJ\right)$. Thus the $\alg{u}(n)$--component of $R$ is given by
\begin{equation}\label{eq:Runbot}
R(X,Y)_{\alg{u}(n)}=\frac{1}{2}\left(R(X,Y)+J\circ R(X,Y)\circ J\right).
\end{equation}
Moreover, the intrinsic torsion, being informally the projection of $-\nabla$ to $\alg{u}(n)^{\bot}$, is given by the formula
\begin{equation}\label{eq:inttorhermitian}
\xi_XY=-\frac{1}{2}J(\nabla_XJ)Y.
\end{equation}
Hence, the characteristic vector field $\chi$ is the following
\begin{equation}\label{eq:characteristichermitian}
\chi=-\frac{1}{2}J({\rm div} J).
\end{equation}

Let us describe the intrinsic torsion with the use of the Nijenhuis tensor $N$ and the K\"ahler form $\Omega$. Recall that
\begin{align*}
N(X,Y) &=[JX,JY]-J[X,JY]-J[JX,Y]-[X,Y]\\
&=(\nabla_XJ)JY-(\nabla_YJ)JX+(\nabla_{JX}J)Y-(\nabla_{JY}J)X
\end{align*}
and
\begin{equation*}
\Omega(X,Y)=g(X,JY).
\end{equation*}
It is a famous theorem by Newlander and Nirenberg that vanishing of the Nijenhuis tensor is equivalent to integrability of $J$, i.e. existence of complex coordinates adapted to $J$. In can be shown \cite{IA} that
\begin{equation}\label{eq:inttorhermitian2}
4g(\xi_XY,Z)=d\Omega(X,Y,JZ)+d\Omega(X,JY,Z)-g(N(Y,Z),X).
\end{equation}
Unfortunately, this shows that $\xi$ has no particular symmetries and, using \eqref{eq:inttorhermitian2} it is hard to give nice interpretations for the symmetrized and skew--symmetrized intrinsic torsion $\xi^{\rm sym}$ and $\xi^{\rm alt}$, respectively. Therefore, it is convenient to consider some restrictions or decomposition of the intrinsic torsion. The space of all possible intrinsic torsions is, in this case, $T^{\ast}M\otimes \alg{u}(n)^{\bot}(TM)$. Decomposing this space into irreducible modules with respect to $U(n)$--action, we get so--called Gray--Hervella classes \cite{GH}
\begin{equation}\label{eq:grayhervellaclasses}
T^{\ast}M\otimes \alg{u}(n)^{\bot}(TM)=\mathcal{W}_1\oplus\mathcal{W}_2\oplus\mathcal{W}_3\oplus\mathcal{W}_4,
\end{equation}
where each class can be characterized as follows:
\begin{align*}
&\mathcal{W}_1: && \textrm{$\xi_XY=-\xi_YX$, in particular, $\chi=0$.}\\
&\mathcal{W}_2: && \textrm{$g(\xi_XY,Z)+g(\xi_ZX,Y)+g(\xi_YZ,X)=0$. Then $\chi=0$.}\\
&\mathcal{W}_3: && \textrm{$\xi_XY=\xi_{JX}(JY)$ and $\chi=0$,}\\
&\mathcal{W}_4: && \textrm{$-4\xi_XY=\theta(Y)X+\theta(JY)JX-g(X,Y)\theta^{\sharp}-g(X,JY)J\theta^{\sharp}$, $\theta\in\Gamma(T^{\ast}M)$.} 
\end{align*}
In the characterization of $\mathcal{W}_4$, $\theta$ is a one--form often called the Lee form. 

The following proposition contains well--known and useful properties of almost Hermitian $\mathcal{W}_1,\ldots,\mathcal{W}_4$ classes.

\begin{prop}[\cite{GH}]\label{prop:GHclassesproperties}
We have the following characterization of Gray--Hervalla classes:
\begin{align*}\label{eq:w1w2w3w4rel}
\mathcal{W}_1\oplus\mathcal{W}_2 &=\{\xi\in T^{\ast}M\otimes \alg{u}(n)^{\bot}(TM)\mid \xi_{JX}JY=-\xi_XY\}\\
\mathcal{W}_2\oplus\mathcal{W}_4 &=\{\xi\in T^{\ast}M\otimes \alg{u}(n)^{\bot}(TM)\mid \xi_{JX}JY=\xi_XY\}.
\end{align*}
Moreover, 
\begin{equation*}
\mathcal{W}_1\oplus\mathcal{W}_2\oplus\mathcal{W}_3=\{\xi\in T^{\ast}M\otimes \alg{u}(n)^{\bot}(TM)\mid \chi=0\}.
\end{equation*}
\end{prop}

By above proposition, if $\xi\in\mathcal{W}_1\oplus\mathcal{W}_2\oplus\mathcal{W}_3$, then formula \eqref{eq:divformain} is a point--wise formula for the $u(n)^{\bot}$--component of the scalar curvature
\begin{equation}\label{eq:pointwisehermitian}
\frac{1}{2}s_{\alg{u}(n)^{\bot}}=|\xi^{\rm alt}|^2-|\xi^{\rm sym}|^2.
\end{equation}

The left--hand--side has a nice interpretation, which is valid for all Gray--Hervella classes. Define Ricci and $\ast$--Ricci tensors by
\begin{equation*}
{\rm Ric}(X,Y)=\sum_i g(R(X,e_i)e_i,Y),\quad {\rm Ric}^{\ast}(X,Y)=g(R(X,e_i)Je_i,JY).
\end{equation*}
These induce, taking traces, scalar curvatures
\begin{equation*}
s=\sum_i {\rm Ric}(e_i,e_i),\quad s^{\ast}=\sum_i {\rm Ric}^{\ast}(e_i,e_i).
\end{equation*}
Then, by \eqref{eq:Runbot}
\begin{equation}\label{eq:sunbot}
\begin{split}
s_{\alg{u}(n)^{\bot}} &=\sum_{i,j}g(R(e_j,e_i)_{\alg{u}(n)^{\bot}}e_i,e_j)\\
&=\frac{1}{2}\sum_{i,j} g(R(e_j,e_i)e_i,e_j)+g(JR(e_j,e_i)Je_i,e_j)\\
&=\frac{1}{2}(s-s^{\ast}).
\end{split}
\end{equation}

Now, we will discus relations between elements in the divergence formula \eqref{eq:divformain} in each pure class $\mathcal{W}_i$ separately. We will proceed by studying quadratic invariants of the $U(n)$--representation on the space of intrinsic torsions $T^{\ast}M\otimes\alg{u}(n)^{\bot}(TM)$ \cite{GH}:
\begin{align*}
& i_1=\sum_{i,j,k}\alpha(e_i,e_j,e_k)^2 && i_2=\sum_{i,j,k} \alpha(e_i,e_j,e_k)\alpha(e_j,e_i,e_k)\\
& i_3=\sum_{i,j,k}\alpha(e_i,e_j,e_k)\alpha(Je_i,Je_j,e_k) && i_4=\sum_{i,j,k}\alpha(e_i,e_i,e_k)\alpha(e_j,e_j,e_k),
\end{align*}
where $\alpha(X,Y,Z)=g(\xi_XY,Z)$. Notice that
\begin{equation*}
i_1=|\xi|^2,\quad i_2=|\xi^{\rm sym}|^2-|\xi^{\rm alt}|^2,\quad i_4=|\chi|^2.
\end{equation*}
Thus the divergence formula \eqref{eq:divformain} may be rewritten, using these invariants and formula \eqref{eq:sunbot} as follows
\begin{equation}\label{eq:divforhermitianinvariants}
{\rm div}\chi=-\frac{1}{4}(s-s^{\ast})-i_2+i_4.
\end{equation}

It is not hard to see by definitions of each pure class $\mathcal{W}_i$ and Proposition \ref{prop:GHclassesproperties} that the following fact holds.

\begin{prop}\label{prop:hermitianinvariants}
Quadratic invariants characterize pure Gray--Hervella classes $\mathcal{W}_i$ as listed in a Table 1. In a table $i_j^{(k)}$ denotes an invariant $i_j$ considered for a class $\mathcal{W}_k$.
\end{prop}

\begin{table}[!h]
\caption{Quadratic invariants for Gray--Hervella classes $\mathcal{W}_k$}
\begin{tabular}{|c|ccc|}
\hline
$\mathcal{W}_1$ & $i_4^{(1)}=0$, & $i_3^{(1)}=-i_1^{(1)}$, & $i_2^{(1)}=-i_1^{(1)}$ \\
\hline
$\mathcal{W}_2$ & $i_4^{(2)}=0$, & $i_3^{(2)}=-i_1^{(2)}$, & $i_2^{(2)}=\frac{1}{2}i_1^{(2)}$ \\
\hline
$\mathcal{W}_3$ & $i_4^{(3)}=0$, & $i_3^{(3)}=i_1^{(3)}$, & $i_2^{(3)}=0$ \\
\hline
$\mathcal{W}_4$ & $i_4^{(4)}=\frac{1}{2}(n-1)i_1^{(4)}$, & $i_3^{(4)}=i_1^{(4)}$, & $i_2^{(4)}=0$ \\
\hline
\end{tabular} 
\end{table}

Using Proposition \ref{prop:hermitianinvariants} and relation \eqref{eq:divforhermitianinvariants} we can derive some useful relations for each pure class. These relations are well-known (see the references listed in the Proposition below). Let us enlarge on this. In analogy to Gray \cite{AG2} we consider the following curvature condition
\begin{equation*}
\mathcal{SC}:\quad s=s^{\ast}.
\end{equation*}
Notice that the class $\mathcal{SC}$ contains Gray class $\mathcal{G}_1$, which by definition, denotes almost Hermitian structures for which the curvature tensor satisfies
\begin{equation*}
\mathcal{G}_1:\quad R(X,Y,Z,W)=R(X,Y,JZ,JW).
\end{equation*}

\begin{prop}\label{prop:curvaturerelations}
The following relations hold.
\begin{enumerate}
\item (compare \cite{AG,FF}) For a $\mathcal{W}_1$ structure $s-s^{\ast}=|\nabla J|^2$. In particular, there is no nearly K\"ahler non--K\"ahler structure satisfying $\mathcal{SC}$ condition.
\item (compare \cite{SK,FF}) For a $\mathcal{W_2}$ structure $2(s-s^{\ast})=-|\nabla J|^2$. In particular, there is no almost K\"ahler non--K\"ahler structure satisfying $\mathcal{SC}$ condition.
\item Any $\mathcal{W}_3$ structure satisfies $\mathcal{SC}$ condition.
\item (compare \cite{IV2,FF0,FF}) For a $\mathcal{W}_4$ structure with a Lee form $\theta$
\begin{equation*}
(n-1){\rm div}\theta^{\sharp}=-(s-s^{\ast})+(n-1)^2|\theta|^2.
\end{equation*}
In particular, there is no locally conformally K\"ahler non--Kahler structure defined on a closed manifold which satisfies $\mathcal{SC}$ condition.
\end{enumerate}
\end{prop}
\begin{proof}
It suffices to apply \eqref{eq:divformain}, Proposition \ref{prop:hermitianinvariants} and use the fact that $|\xi|^2=\frac{1}{4}|\nabla J|^2$. For a $\mathcal{W}_4$ case, notice, that
\begin{equation*}
\chi=\frac{n-1}{2}\theta^{\sharp},\quad i^{(4)}_4=\frac{(n-1)^2}{4}|\theta|^2.\qedhere
\end{equation*}
\end{proof}

Now, we show that the main integral formula \eqref{eq:intformain} in an almost Hermitian case is equivalent to a Bor--Lamoneda formula \cite{BL}. Decompose $\xi$ and $\chi$ with respect to the Gray--Hervella classes as follows
\begin{equation*}
\xi=\xi^1+\xi^2+\xi^3+\xi^4,\quad \chi=\chi^1+\chi^2+\chi^3+\chi^4,
\end{equation*}
i.e., $\chi_k=\sum_i \xi^k_{e_i}e_i$, and let
\begin{equation*}
E_k=|\chi^k|^2+|\xi^{k,\rm alt}|^2-|\xi^{k,\rm sym}|^2=-i_2^{(k)}+i_4^{(k)}.
\end{equation*} 
It can be shown that
\begin{equation*}
i_j=\sum_k i_j^{(k)},\quad j=1,2,3,4. 
\end{equation*}
Thus, by above considerations (see also Remark \ref{rem:invariants}), we have
\begin{equation*}
E_1=|\xi^1|^2,\quad E_2=-\frac{1}{2}|\xi^2|^2,\quad E_3=0,\quad E_4=\frac{1}{2}(n-1)|\xi^4|^2.
\end{equation*}
Hence
\begin{equation}\label{eq:BLformula}
{\rm div}\chi=|\xi^1|^2-\frac{1}{2}|\xi^2|^2+\frac{n-1}{2}|\xi^4|^2-\frac{1}{4}(s-s^{\ast}),
\end{equation}
which implies the integral formula by Bor and Hern\'{a}ndez Lamoneda \cite{BL} (assuming $M$ is closed)
\begin{equation}\label{eq:BHLformula}
\int_M (2|\xi^1|^2-|\xi^2|^2+(n-1)|\xi^4|^2)\,{\rm vol}_M
=\frac{1}{2}\int_M s-s^{\ast}\,{\rm vol}_M.
\end{equation}

\subsection{Special almost Hermitian structures}

Assume $(M,g,J)$ is an almost Hermitian manifold equipped with a complex volume form $\Psi=\psi_++i\psi_-$ such that $\skal{\Psi}{\Psi}_{\mathbb{C}}=1$, where the inner product is a natural extension of an inner product for real valued forms. This structure defines reduction of a structure group to special unitary group $SU(n)$, hence a $SU(n)$--structure. On the level of Lie algebras we have
\begin{equation*}
\alg{so}(2n)=\alg{su}(n)\oplus(\alg{u}(n)^{\bot}\oplus\mathbb{R}),
\end{equation*}
since $\alg{u}(n)=\alg{su}(n)\oplus\mathbb{R}$. For an element $A\in\alg{u}(n)$ let
\begin{equation*}
A=\left(\begin{array}{cc} A_0 & -A_1 \\ A_1 & A_0\end{array}\right)\in\alg{so}(2n).
\end{equation*}
Then $A\in\alg{su}(n)$ if and only if $A\in\alg{u}(n)$ and ${\rm tr}A_1=0$. Notice that 
\begin{equation*}
{\rm tr}A_1=\frac{1}{2}\sum_i g(Ae_i,Je_i)=-\frac{1}{2}{\rm tr}(AJ).
\end{equation*} 
Thus, the orthogonal projection from $\alg{so}(2n)$ to $\alg{su}(n)^{\bot}=\alg{u}(n)^{\bot}\oplus\mathbb{R}$ equals
\begin{equation*}
A\mapsto \frac{1}{2}(A+JAJ)-\frac{1}{2n}({\rm tr} AJ)J.
\end{equation*}
The intrinsic torsion $\xi$ equals $\xi=\xi^{U(n)}+\eta$ \cite{BL,CS,FMC}, where $\xi^{U(n)}$ is the intrinsic torsion of related $U(n)$--structure and
\begin{equation*}
\eta_XY=-\frac{1}{2n}\sum_i g(\xi_XJe_i,e_i)JY.
\end{equation*}
Define a one--form $\eta$ by the relation $\eta_XY=\eta(JX)JY$. This convention will appear to be useful. Denote the class in the space of all possible intrinsic torsions induced by $\eta$ by $\mathcal{W}_5$. 

Split $s_{\alg{su}(n)^{\bot}}$ into $s_{\alg{u}(n)^{\bot}}$ and $s_{\mathbb{R}}$ with respect to the decomposition $\alg{su}(n)^{\bot}=\alg{u}(n)^{\bot}\oplus\mathbb{R}$.

\begin{prop}\label{prop:SUn}
On a closed $SU(n)$--structure $(M,g,J)$ with the $\mathcal{W}_5$-component induced by the $1$--form $\eta$ we have the following integral formula
\begin{equation}\label{eq:intforSUn}
8\int_M \eta(\chi^{U(n)})\,{\rm vol}_M=\int_M s_{\mathbb{R}}-s^{\rm alt}_{\alg{su}(n)^{\bot}}\,{\rm vol}_M.
\end{equation}
In particular, if $(M,g,J)$ is of Gray--Hervella class $\mathcal{W}_1\oplus\mathcal{W}_2\oplus\mathcal{W}_3$ treated as $U(n)$--structure, then $\int_M s_{\mathbb{R}}=\int_M s_{\alg{su}(n)^{\bot}}^{\rm alt}$.
\end{prop}
\begin{proof}
\begin{equation*}
\chi=\chi^{U(n)}+\sum_i \eta(Je_i)Je_i=\chi^{U(n)}+\eta^{\sharp}
\end{equation*}
and
\begin{align*}
\sum_{i,j}g(\xi_{e_i}e_j,\xi_{e_j}e_i) &=\sum_{i,j}g(\xi^{U(n)}_{e_i}e_j,\xi^{U(n)}_{e_j}e_i)+
2\sum_{i,j}g(\xi^{U(n)}_{e_i}e_j,\eta(Je_j)Je_i)\\
&+\sum_{i,j} \eta(Je_i)\eta(Je_j)g(Je_j,Je_i)\\
&=\sum_{i,j}g(\xi^{U(n)}_{e_i}e_j,\xi^{U(n)}_{e_j}e_i)-2\eta(\chi^{U(n)})+|\eta^{\sharp}|^2.
\end{align*}
Thus
\begin{align*}
{\rm div}\chi &=-\frac{1}{2}s_{\alg{su}(n)^{\bot}}+\frac{1}{2}s^{\rm alt}_{\alg{su}(n)^{\bot}}+|\chi|^2-\sum_{i,j}g(\xi_{e_i}e_j,\xi_{e_j}e_i)\\
&=-\frac{1}{2}s_{\alg{su}(n)^{\bot}}+\frac{1}{2}s^{\rm alt}_{\alg{su}(n)^{\bot}}+|\chi^{U(n)}|^2+2\eta(\chi^{U(n)})+|\eta^{\sharp}|^2\\
&-\sum_{i,j}g(\xi^{U(n)}_{e_i}e_j,\xi^{U(n)}_{e_j}e_i)+2\eta(\chi^{U(n)})-|\eta^{\sharp}|^2\\
&={\rm div}\chi^{U(n)}+\frac{1}{2}s^{\rm alt}_{\alg{su}(n)^{\bot}}-\frac{1}{2}s_{\alg{su}(n)^{\bot}}+\frac{1}{2}s_{\alg{u}(n)^{\bot}}+4\eta(\chi^{U(n)}).
\end{align*}
Hence
\begin{equation}\label{eq:diveta}
{\rm div}\eta^{\sharp}=-\frac{1}{2}s_{\mathbb{R}}+\frac{1}{2}s^{\rm alt}_{\alg{su}(n)^{\bot}}+4\eta(\chi^{U(n)}).
\end{equation}
Assuming $M$ is closed and applying the Stokes theorem we get \eqref{eq:intforSUn}.
\end{proof}

The values of $s_{\mathbb{R}}$ and $s_{\alg{su}(n)^{\bot}}^{\rm alt}$ can be computed explicitly, which gives an alternative version of formula \eqref{eq:intforSUn}. Firstly, introduce two components of the intrinsic torsion, $\xi^{U(n),12}\in\mathcal{W}_1\oplus\mathcal{W}_2$ and $\xi^{U(n),34}\in\mathcal{W}_3\oplus\mathcal{W}_4$.

\begin{cor}\label{cor:SUn}
On a closed $SU(n)$--structure $(M,g,J)$ with the $\mathcal{W}_5$-component induced by the $1$--form $\eta$ we have the following integral formula
\begin{equation*}
\int_M s^{\ast}\,{\rm vol}_M=\int_M (|\xi^{U(n),12}|^2-|\xi^{U(n),34}|^2)\,{\rm vol}_M+4n\int_M\eta(\chi^{U(n)})\,{\rm vol}_M. 
\end{equation*}
In particular, if $\xi^{U(n)}\in\mathcal{W}_1\oplus\mathcal{W}_2\oplus\mathcal{W}_3\oplus\mathcal{W}_5$, then
\begin{equation*}
\int_M s^{\ast}\,{\rm vol}_M=\int_M |\xi^{U(n),1}|^2+|\xi^{U(n),2}|^2-|\xi^{U(n),3}|^2\,{\rm vol}_M.
\end{equation*}
\end{cor}
\begin{proof}
We have
\begin{align*}
s_{\mathbb{R}} &=\sum_{i,j}g(R(e_i,e_j)_{\mathbb{R}}e_j,e_i)\\
&=-\frac{1}{2n}\sum_{i,j}{\rm tr}(R(e_i,e_j)J)g(Je_j,e_i)\\
&=-\frac{1}{2n}\sum_{i,j}g(R(Je_j,e_j)Je_i,e_i)\\
&=\frac{1}{2n}\sum_{i,j}(g(R(Je_i,Je_j)e_j,e_i)+g(R(e_j,Je_i)Je_j,e_i))\\
&=\frac{1}{n}s^{\ast}.
\end{align*} 
To compute $s_{\alg{su}(n)^{\bot}}^{\rm alt}$, it is convenient to determine the component $[\alg{su}(n)^{\bot},\alg{su}(n)^{\bot}]_{\alg{su}(n)^{\bot}}$. For $A=A_0+\lambda J$ and $B=B_0+\mu J$, where $A_0,B_0\in \alg{u}(n)^{\bot}$, by the relation $[\alg{u}(n)^{\bot},\alg{u}(n)^{\bot}]\subset\alg{u}(n)$, we have
\begin{align*}
[A,B]_{\alg{su}(n)^{\bot}} &=[A_0,B_0]_{\mathbb{R}}+\lambda[J,B_0]+\mu[A_0,J]\\
&=-\frac{1}{n}{\rm tr}(A_0B_0J)J+2(\mu A_0-\lambda B_0)J.
\end{align*} 
Hence,
\begin{align*}
s^{\rm alt}_{\alg{su}(n)^{\bot}} &=\sum_{i,j}\left(-\frac{1}{n}{\rm tr}(\xi_{e_i}\xi_{e_j}J)g(Je_j,e_i)+2g(\eta(Je_j)\xi_{e_i}Je_j-\eta(Je_i)\xi_{e_j}Je_j,e_i)\right)\\
&=-\frac{1}{n}\sum_{i,j}g(\xi_{Je_j}\xi_{e_j}Je_i,e_i)+2\sum_{i,j}\left(\eta(Je_j)g(\xi_{e_i}Je_j,e_i)-\eta(Je_i)g(\xi_{e_j}Je_j,e_i)\right)\\
&=\frac{1}{n}\sum_{i,j}g(\xi_{e_j}Je_i,\xi_{Je_j}e_i)-4\eta(\chi^{U(n)})\\
&=-\frac{1}{n}\sum_{i,j}g(\xi_{e_j}e_i,\xi_{Je_j}Je_i)-4\eta(\chi^{U(n)}).
\end{align*}
By Proposition \ref{prop:GHclassesproperties} we get
\begin{equation*}
\sum_{i,j}g(\xi_{e_j}e_i,\xi_{Je_j}Je_i)=-|\xi^{U(n),12}|^2+|\xi^{U(n),34}|^2.
\end{equation*}
now, it suffices to apply \eqref{eq:intforSUn}.
\end{proof}

\begin{rem}
The above integral formula, however formulated in a different way, can be found in \cite{BL}. Let us be more precise. In \cite{BL} authors state some consequences of their formula for almost Hermitian structures with vanishing first Chern class $c_1(M)$. Let us derive the first Chern class in our setting. It is known \cite{AG} that the first Chern form $\gamma$ is given by
\begin{equation*}
8\pi\gamma=-\varphi+2\psi,
\end{equation*}  
where
\begin{equation*}
\varphi(X,Y)={\rm tr}((\nabla_XJ)(\nabla_YJ)),\quad \psi(X,Y)={\rm tr}(R(X,Y)\circ J).
\end{equation*}
It is not hard to see, that
\begin{equation*}
\varphi(X,Y)=4\sum_i g(\xi_X Je_i,\xi_Ye_i),\quad \psi(X,Y)=-2{\rm Ric}^{\ast}(X,JY).
\end{equation*} 
Thus, using the same arguments as before Corollary \ref{cor:SUn}, we get
\begin{equation*}
2\pi{\rm tr}^{\ast}\gamma=|\xi^{U(n),34}|^2-|\xi^{U(n),12}|^2+s^{\ast},
\end{equation*}
where ${\rm tr}^{\ast}\gamma=\sum_i \gamma(e_i,Je_i)$. Notice that vanishing of the first Chern class, i.e., $\gamma=d\alpha$ for some $1$--form, is equivalent to the fact that $\int_M {\rm tr}^{\ast}\gamma=0$. Thus by Corollary \ref{cor:SUn}, $c_1(M)=0$ if and only if $\int_M \eta(\chi^{U(n)})=0$. Finally, note that by \eqref{eq:divergences} we have
\begin{equation*}
\int_M \eta(\chi^{U(n)})\,{\rm vol}_M=-\int_M {\rm div}^G\eta^{\sharp}\,{\rm vol}_M.
\end{equation*}
\end{rem}

\begin{cor}
Consider an $SU(n)$--structure $(M,g,J)$ which is of type $\mathcal{W}_1\oplus\mathcal{W}_5$. Then
\begin{equation*}
\int_M s\,{\rm vol}_M=5\int_M s^{\ast}\,{\rm vol}_M.
\end{equation*}
\end{cor}
\begin{proof}
By Corollary \ref{cor:SUn} we have 
\begin{equation*}
\int_M s^{\ast}=\int_M |\xi^{U(n)}|^2.
\end{equation*}
It suffices to notice that by \eqref{eq:sunbot} and \eqref{eq:pointwisehermitian} we have $\frac{1}{4}(s-s^{\ast})=|\xi^{U(n)}|^2$. 
\end{proof}

\subsection{Almost contact metric structures}

Let $(M,g)$ be a $(2n+1)$--dimensional manifold together with a $1$--form $\eta$ (and its dual unit vector field $\zeta$) and $\varphi\in{\rm End}(TM)$ such that
\begin{equation}\label{eq:contactdef}
\varphi^2X=-X+\eta(X)\zeta,\quad g(\varphi X,\varphi Y)=g(X,Y)-\eta(X)\eta(Y).
\end{equation}
Notice that $\varphi$ defines almost complex structure which is $g$--orthogonal on the distribution ${\rm ker}\eta$. Thus we get $U(n)\times 1$--structure. On the level of Lie algebras we have
\begin{equation*}
\alg{so}(2n+1)=\alg{u}(n)\oplus\alg{u}(n)^{\bot},
\end{equation*}
where $\alg{u}(n)^{\bot}$ is isomorphic to the space of block matrices of the form
\begin{equation}\label{unbotodd}
\left(\begin{array}{cc} B & a \\ -a^{\top} & 0 \end{array}\right),\quad B\in\alg{u}(n)^{\bot}\subset\alg{so}(2n),\quad a\in\mathbb{R}^{2n}.
\end{equation}
Since here $\zeta=e_{2n+1}$, $\varphi$ is a natural complex structure on $\mathbb{R}^{2n}$ and zero on $\zeta$, it is easy to see that the orthogonal projection from $\alg{so}(2n+1)$ onto $\alg{u}(n)^{\bot}$ equals
\begin{equation}\label{eq:projcontact}
A\mapsto \frac{1}{2}(A+\varphi A \varphi+\zeta^{\top}A\otimes \zeta+\zeta^{\top}\otimes A\zeta).
\end{equation}
Rewriting this formula with the use of the one--form $\eta (\equiv \zeta^{\top})$, the intrinsic torsion satisfies the following relation
\begin{equation}\label{eq:inttorcontactrel}
\xi_XY=\varphi(\xi_X\varphi Y)+\eta(\xi_XY)\zeta+\eta(Y)\xi_X\zeta.
\end{equation}
This, moreover, implies the formula for the intrinsic torsion \cite{GMC2}
\begin{equation}\label{eq:inttorcontact}
\xi_XY=\frac{1}{2}(\nabla_X\varphi)\varphi Y+\frac{1}{2}(\nabla_X\eta)Y\cdot \zeta-\eta(Y)\nabla_X\zeta.
\end{equation}
By \eqref{eq:contactdef} it follows that
\begin{equation*}
\varphi(\nabla_X\varphi)Y=-(\nabla_X\varphi)\varphi Y+(\nabla_X\eta)Y\cdot \zeta+\eta(Y)\nabla_X\zeta.
\end{equation*}
Thus, we may write the intrinsic torsion in an alternative way \cite{GMC2}
\begin{equation}\label{eq:inttorcontactalternative}
\xi_XY=-\frac{1}{2}\varphi(\nabla_X\varphi)Y+(\nabla_X\eta)Y\cdot \zeta-\frac{1}{2}\eta(Y)\nabla_X\eta.
\end{equation}
Hence, the characteristic vector field in this case equals
\begin{equation}\label{eq:charvfcontact}
\chi=-\frac{1}{2}\varphi({\rm div}\varphi)+({\rm div}\zeta)\zeta-\frac{1}{2}\nabla_{\zeta}\zeta.
\end{equation}

The condition $\xi\in T^{\ast}M\otimes\alg{u}(n)^{\bot}(TM)$ is equivalent to the relation \eqref{eq:inttorcontactrel}. Decomposing the space $T^{\ast}M\otimes\alg{u}(n)^{\bot}(TM)$ into irreducible $U(n)\times 1$--modules, we get twelve classes $\mathcal{C}_1,\ldots,\mathcal{C}_{12}$ \cite{CG}. First four are isomorphic to Gray--Hervella classes $\mathcal{W}_1,\ldots,\mathcal{W}_4$. 

\begin{rem}\label{rem:correspondence}
Note that in \cite{CG} types of almost contact metric structures, i.e. irreducible modules of $T^{\ast}M\otimes\alg{u}(n)^{\bot}(TM)$, were classified with respect to $\alpha(X,Y,Z)=(\nabla_X\Phi)(Y,Z)$, where $\Phi(X,Y)=g(X,\varphi Y)$. It is well known that this is equivalent to considering the intrinsic torsion as a map $\beta(X,Y,Z)=g(\xi_XY,Z)$. The correspondence follows from the fact that $\nabla_X\Phi=\xi_X\Phi$, since $\nabla^{U(n)\times 1}\Phi=0$. This implies a direct relation
\begin{equation}\label{eq:contactcorrespondence}
\alpha(X,Y,Z)=\beta(X,Y,\varphi Z)-\beta(X,Z,\varphi Y).
\end{equation}
Note that we should be careful with studying irreducible modules $\mathcal{C}_1,\ldots,\mathcal{C}_{12}$, since the correspondence $\alpha\leftrightarrow\beta$ interchanges some of the modules, which is underlined in the table below.
\begin{table}[!h]
\caption{Module correspondence via $\alpha\leftrightarrow\beta$}
\begin{tabular}{|c|c|c|c|c|c|c|c|c|c|c|c|c|}
\hline
$\alpha$ & $\mathcal{C}_1$ & $\mathcal{C}_2$ & $\mathcal{C}_3$ & $\mathcal{C}_4$ & $\mathcal{C}_5$ & $\mathcal{C}_6$ & $\mathcal{C}_7$ & $\mathcal{C}_8$ & $\mathcal{C}_9$ & $\mathcal{C}_{10}$ & $\mathcal{C}_{11}$ & $\mathcal{C}_{12}$ \\
\hline
$\beta$ & $\mathcal{C}_1$ & $\mathcal{C}_2$ & $\mathcal{C}_3$ & $\mathcal{C}_4$ & $\mathcal{C}_6$ & $\mathcal{C}_5$ & $\mathcal{C}_8$ & $\mathcal{C}_7$ & $\mathcal{C}_9$ & $\mathcal{C}_{10}$ & $\mathcal{C}_{11}$ & $\mathcal{C}_{12}$ \\
\hline
\end{tabular} 
\end{table}
\end{rem}

Let us describe these spaces in more detail. Put
\begin{equation*}
\mathcal{D}_1=\mathcal{C}_1\oplus\ldots\oplus\mathcal{C}_4,\quad
\mathcal{D}_2=\mathcal{C}_5\oplus\ldots\oplus\mathcal{C}_{11},\quad
\mathcal{D}_3=\mathcal{C}_{12}.
\end{equation*}
Each of above spaces is characterized as follows \cite{CG}:
\begin{enumerate}
\item {\bf Class} $\mathcal{D}_1$: $\xi_{\zeta}Y=\xi_X\zeta=0$. Applying formluas \eqref{eq:inttorcontact} and \eqref{eq:inttorcontactalternative} we obtain $\nabla\zeta=0$ and hence $\chi=-\frac{1}{2}\varphi({\rm div}\varphi)$, as expected, since in this case, being not very precise, $\xi$ is the intrinsic torsion on the almost Hermitian structure ${\rm ker}\eta$.
\item {\bf Class} $\mathcal{D}_2$: $\xi_XY=\eta(X)\xi_{\zeta}Y+\eta(Y)\xi_X\zeta+\eta(\xi_XY)\zeta$. Taking $X=Y=\zeta$ we get $\xi_{\zeta}\zeta=-\eta(\xi_{\zeta}\zeta)\zeta$ and $\eta$ to both sides we get $\xi_{\zeta}\zeta=0$. Taking $X=Y=e_i$ we obtain 
\begin{equation*}
{\rm pr}_{{\rm ker}\eta}\chi=0.
\end{equation*}
Hence, $\chi\in{\rm span}\,\zeta$, so by \eqref{eq:charvfcontact}, $\chi=({\rm div}\zeta)\zeta$. Moreover, it is easy to see that for $X,Y\in{\rm ker}\eta$
\begin{equation*}
\xi_XY\in{\rm span}\,\zeta,\quad\xi_{\zeta}Y\in{\rm ker}\eta,\quad \xi_X\zeta\in{\rm ker}\eta,\quad\xi_{\zeta}\zeta=0.
\end{equation*}
Thus, by the formulas \eqref{eq:inttorcontact} and \eqref{eq:inttorcontactalternative} for the intrinsic torsion  we get ($X,Y\in{\rm ker}\eta$)
\begin{align*}
\xi_XY &=(\nabla_X\eta)Y\cdot\zeta,\\
\xi_{\zeta}Y &=\frac{1}{2}(\nabla_{\zeta}\varphi)\varphi Y,\\
\xi_X\zeta &=-\nabla_X\zeta.
\end{align*} 
\item {\bf Class} $\mathcal{D}_3$: $\xi_XY=\eta(X)\eta(Y)\xi_{\zeta}\zeta+\eta(X)\eta(\xi_{\zeta}Y)\zeta$. Therefore, for $X,Y\in{\rm ker}\eta$
\begin{equation*}
\xi_XY=0,\quad \xi_X\zeta=0,\quad\xi_{\zeta}Y\in{\rm span}\,\zeta,\quad\xi_{\zeta}\zeta\in{\rm ker}\eta.
\end{equation*}
Hence, $\chi=\xi_{\zeta}\zeta=-\nabla_{\zeta}\zeta$ and $|\chi|^2+|\xi^{\rm alt}|^2-|\chi^{\rm sym}|^2=0$.
\end{enumerate}

Analogously as in an almost Hermitian case, let us relate divergence and integral formulas \eqref{eq:divformain}, \eqref{eq:intformain} to quadratic invariants of almost contact metric structure. In this case the space of quadratic invariants is generated by $18$ invariants \cite{CG}. For our purposes we need only some of them:
\begin{align*}
& i_1=\sum_{i,j,k}\alpha(e_i,e_j,e_k)^2, && i_2=\sum_{i,j,k}\alpha(e_i,e_j,e_k)\alpha(e_j,e_i,e_k)^2,\\
& i_4=\sum_{i,j,k}\alpha(e_i,e_i,e_k)\alpha(e_j,e_j,e_k), && i_5=\sum_{j,k}\alpha(\zeta,e_j,e_k)^2,\\
& i_6=\sum_{i,j,k}\alpha(e_i,\zeta,e_k)^2, && i_7=\sum_{j,k}\alpha(\zeta,e_j,e_k)\alpha(e_j,\zeta,e_k),\\
& i_8=\sum_{i,j}\alpha(e_i,e_j,\zeta)\alpha(e_j,e_i,\zeta), && i_{10}=\sum_{i,j}\alpha(e_i,e_i,\zeta)\alpha(e_j,e_j,\zeta),\\
& i_{12}=\sum_{i,j}\alpha(e_i,e_j,\zeta)\alpha(\varphi(e_j),\varphi(e_i),\zeta), && i_{14}=\sum_{i,j}\alpha(e_i,\varphi(e_i),\zeta)\alpha(e_j,\varphi(e_j),\zeta),\\ 
& i_{16}=\sum_k\alpha(\zeta,\zeta,e_k)^2, && i_{17}=\sum_{i,k}\alpha(e_i,e_i,e_k)\alpha(\zeta,\zeta,e_k),
\end{align*} 
where $\alpha(X,Y,Z)=g(\xi_XY,Z)$ and $i,j=1,\ldots,2n$. Notice that
\begin{align*}
|\xi|^2 &=i_1+i_5+2i_6+2i_{16},\\
|\chi|^2 &=i_4+i_{10}+i_{16},\\
|\xi^{\rm sym}|^2-|\xi^{\rm alt}|^2 &=i_2+i_7+i_8+i_{16}.
\end{align*}

Let us now compute scalar curvature 'components'. Define the $\ast$--scalar curvature as follows
\begin{equation*}
s^{\ast}={\rm tr}\,{\rm Ric}^{\ast}=\sum_{i,j}g(R(e_i,e_j)\varphi e_j,\varphi e_i).
\end{equation*}

\begin{lem}\label{lem:scalarcurvaturecompcontact}
The following relations hold
\begin{align*}
s^{\rm alt}_{\alg{u}(n)^{\bot}} &=\frac{1}{2}(i_8-i_{10}+i_{12}+i_{14})+2(i_7-i_{17})\quad\textrm{and}\\ 
s_{\alg{u}(n)^{\bot}} &=\frac{1}{2}(s-s^{\ast})+{\rm Ric}(\zeta,\zeta).
\end{align*}
\end{lem}
\begin{proof}
Denote by $(B,a)$ an element in $\alg{u}(n)^{\bot}$ of the form \eqref{unbotodd}. Note that $a=(B,a)e_{2n+1}\in\mathbb{R}^{2n}\subset\mathbb{R}^{2n+1}$.  Then,
\begin{equation*}
[(B,a),(\tilde{B},\tilde{a})]=(-a\wedge \tilde{a},B\tilde{a}-\tilde{B}a),
\end{equation*}
where $a\wedge\tilde{a}$ is an endomorphism of $\mathbb{R}^{2n}$ given by $(a\wedge \tilde{a})v=\skal{\tilde{a}}{v}a-\skal{a}{v}\tilde{a}$. Thus
\begin{equation*}
[(B,a),(\tilde{B},\tilde{a})]_{\alg{u}(n)^{\bot}}
=(-\frac{1}{2}(a\wedge\tilde{a}+\varphi_0(a\wedge\tilde{a})\varphi_0,B\tilde{a}-\tilde{B}a),
\end{equation*}
where $\varphi_0$ is a restriction of $\varphi$ to $\mathbb{R}^{2n}$, thus defining an almost complex structure. 

By these considerations we are ready to compute $s_{\alg{u}(n)^{\bot}}^{\rm alt}$. We have (here $i,j=1,\ldots,2n$)
\begin{align*}
s^{\rm alt}_{\alg{u}(n)^{\bot}} &=-\frac{1}{2}\sum_{i,j}g((\xi_{e_i}z\wedge\xi_{e_j}z+\varphi(\xi_{e_i}z\wedge\xi_{e_j}z)\varphi)e_j,e_i)+2\sum_i g(\xi_{e_i}\xi_{\zeta}\zeta-\xi_{\zeta}\xi_{e_i}\zeta,e_i)\\
&=-\frac{1}{2}\sum_{i,j}\bigg(g(\xi_{e_i}e_i,\zeta)g(\xi_{e_j}e_i,\zeta)-g(\xi_{e_i}e_j,\zeta)g(\xi_{e_j}e_i,\zeta)\\
&-g(\xi_{e_i}\varphi(e_i),\zeta)g(\xi_{e_j}\varphi(e_j),\zeta)+g(\xi_{e_i}\varphi(e_j),\zeta)g(\xi_{e_j}\varphi(e_i),\zeta)\bigg)\\
&+2\sum_i (-g(\xi_{\zeta}\zeta,\xi_{e_i}e_i)+g(\xi_{e_i}\zeta,\xi_{\zeta}e_i))\\
&=\frac{1}{2}(-i_{10}+i_8+i_{14}+i_{12})+2(-i_{17}+i_7).
\end{align*}
For $s_{\alg{u}(n)^{\bot}}$ by \eqref{eq:projcontact} we have
\begin{align*}
s_{\alg{u}(n)^{\bot}} &=\sum_{i,j}g(R(e_i,e_j)_{\alg{u}(n)^{\bot}}e_j,e_i)\\
&=\frac{1}{2}(s-s^{\ast})+\frac{1}{2}\sum_{i,j}(\eta(R(e_i,e_j)e_j)\eta(e_i)+\eta(e_j)g(R(e_i,e_j)\zeta,e_i)\\
&=\frac{1}{2}(s-s^{\ast})+\sum_i(g(R(\zeta,e_j)e_j,\zeta)\\
&=\frac{1}{2}(s-s^{\ast})+{\rm Ric}(\zeta,\zeta).\qedhere
\end{align*}
\end{proof}

By Lemma \ref{lem:scalarcurvaturecompcontact} and above considerations we may rewrite formula \eqref{eq:divformain} as
\begin{equation}\label{eq:divforinvariantscontact}
{\rm div}\chi=-i_2+i_4-\frac{3}{4}i_8+\frac{3}{4}i_{10}+\frac{1}{4}i_{12}-\frac{1}{4}i_{14}-i_{17}-\frac{1}{4}(s-s^{\ast})-\frac{1}{2}{\rm Ric}(\zeta,\zeta).
\end{equation} 
Moreover, by a classification of each module $\mathcal{C}_i$ by quadratic invariants \cite[Table I]{CG} we have the following observation.

\begin{prop}\label{prop:contactinvariants}
We have:
\begin{enumerate}
\item the characteristic vector field $\chi$ vanishes if and only if the intrinsic torsion belongs to $\mathcal{C}_1\oplus\mathcal{C}_2\oplus\mathcal{C}_3\oplus\mathcal{C}_6\oplus\mathcal{C}_7\oplus\mathcal{C}_8\oplus\mathcal{C}_9\oplus\mathcal{C}_{10}\oplus\mathcal{C}_{11}$,
\item $|\xi^{\rm sym}|^2-|\xi^{\rm alt}|^2=0$ if and only if the intrinsic torsion belongs to $\mathcal{C}_3\oplus\mathcal{C}_4\oplus\mathcal{C}_{11}$,
\item $s^{\rm alt}_{\alg{u}(n)^{\bot}}=0$ if and only if the intrinsic torsion belongs to $\mathcal{C}_1\oplus\mathcal{C}_2\oplus\mathcal{C}_3\oplus\mathcal{C}_4\oplus\mathcal{C}_7\oplus\mathcal{C}_9\oplus\mathcal{C}_{10}\oplus\mathcal{C}_{11}\oplus\mathcal{C}_{12}$.
\end{enumerate}
\end{prop}

We are ready to interpret, at least in some cases, the formula \eqref{eq:divforinvariantscontact} in a geometric way. Denote by $B^{\eta}$ and $T^{\eta}$ the (symmetric) second fundamental form and integrability tensor of ${\rm ker}\eta$, respectively (see the section on almost product structures). 

\begin{prop}\label{prop:contact1}
Let $(M,g,\varphi,\eta,\zeta)$ be an almost contact metric structure with the intrinsic torsion $\xi$.
\begin{enumerate}
\item If $\xi\in\mathcal{D}_2$, then
\begin{equation*}
{\rm div}(({\rm div}\zeta)\zeta)=({\rm div}\zeta)^2+|T^{\eta}|^2-|B^{\eta}|^2+\frac{1}{2}s^{\rm alt}_{\alg{u}(n)^{\bot}}-\frac{1}{4}(s-s^{\ast})-\frac{1}{2}{\rm Ric}(\zeta,\zeta),
\end{equation*}
or, equivalently,
\begin{equation*}
{\rm div}(\nabla_{\zeta}\zeta)=\frac{1}{2}s^{\rm alt}_{\alg{u}(n)^{\bot}}+\frac{1}{2}{\rm Ric}(\zeta,\zeta)-\frac{1}{4}(s-s^{\ast}).
\end{equation*}
If additionally, $M$ is closed, then the following integral formula holds
\begin{equation*}
\int_M {\rm Ric}(\zeta,\zeta)\,{\rm vol}_M=\frac{1}{2}\int_M s-s^{\ast}-2s^{\rm alt}_{\alg{u}(n)^{\bot}}\,{\rm vol}_M=\int_M ({\rm div}\zeta)^2+|T^{\eta}|^2-|B^{\eta}|^2\,{\rm vol}_M.
\end{equation*}
\item If $\xi\in\mathcal{C}_6\oplus\ldots\oplus\mathcal{C}_{11}$, then
\begin{equation*}
|T^{\eta}|^2-|B^{\eta}|^2=-\frac{1}{2}s^{\rm alt}_{\alg{u}(n)^{\bot}}+\frac{1}{4}(s-s^{\ast})+\frac{1}{2}{\rm Ric}(\zeta,\zeta)
\end{equation*}
\item If $\xi\in\mathcal{C}_{11}$, then $s-s^{\ast}=-2{\rm Ric}(\zeta,\zeta)$ and $\int_M s-s^{\ast}\,{\rm vol}_M=\int_M {\rm Ric}(\zeta,\zeta)\,{\rm vol}_M=0$.
\item If $\xi\in\mathcal{C}_{12}$ then
\begin{equation*}
{\rm div}(\nabla_{\zeta}\zeta)=\frac{1}{4}(s-s^{\ast})+\frac{1}{2}{\rm Ric}(\zeta,\zeta).
\end{equation*}
If, additionally, $M$ is closed, then 
\begin{equation*}
\int_M s-s^{\ast}\,{\rm vol}_M=-2\int_M {\rm Ric}(\zeta,\zeta)\,{\rm vol}.
\end{equation*}
\end{enumerate}
\end{prop}
\begin{proof}
(1) Assume $\xi\in\mathcal{D}_2$. Then, for $X,Y\in{\rm ker}\eta$ we have
\begin{equation*}
\xi^{\rm alt}_XY=-T^{\eta}(X,Y),\quad \xi^{\rm sym}_XY=-B^{\eta}(X,Y).
\end{equation*}
Moreover, by classification of $\mathcal{D}_2$ by quadratic invariants \cite{CG}, we see that 
\begin{equation*}
|\xi^{\rm alt}|^2-|\xi^{\rm sym}|^2=-i_8=|T^{\eta}|^2-|B^{\eta}|^2.
\end{equation*}
Since $\chi=({\rm div}\zeta)\zeta$ using Proposition \ref{prop:Walczakcodim1} the divergence formula \eqref{eq:divformain} takes the first form. The second one and the integral formula follows again by Proposition \ref{prop:Walczakcodim1}.\\
(2) In this case, by Proposition \ref{prop:contactinvariants}, the characteristic vector field vanishes, i.e., ${\rm div}\zeta=0$. Thus, it suffices to apply the first part.\\
(3) This is an immediate consequence of (1), (2) and the fact that $|T^{\eta}|^2-|B^{\eta}|^2$ and $s^{\rm alt}_{\alg{u}(n)^{\bot}}$ vanish (by Proposition \ref{prop:contactinvariants}).\\
(4) By discussion concerning $\mathcal{D}_3$ class, we have $|\chi|^2+|\xi^{\rm alt}|^2-|\xi^{\rm sym}|^2=0$ and $\chi=-\nabla_{\zeta}\zeta$. Moreover, by Proposition \ref{prop:contactinvariants}, $s^{\rm alt}_{\alg{u}(n)^{\bot}}=0$. Hence, the divergence formula \eqref{eq:divformain} and the integral formula \eqref{eq:intformain} simplify to the desired ones. Notice, that we could use the formula \eqref{eq:divforinvariantscontact} and the fact that for $C_{12}$ class the only nonzero invariant is $i_{16}$ \cite[Table I]{CG}.
\end{proof}

At the end, consider the following Janssen--Vanhecke $\mathcal{C}(\alpha)$ condition \cite{JV}:
\begin{align*}
R(X,Y,Z,W) &=R(X,Y,\varphi Z,\varphi W)+\alpha(-g(X,Z)g(Y,W)+g(X,W)g(Y,Z)\\
&+g(X,\varphi Z)g(Y,\varphi W)-g(X,\varphi W)g(Y,\varphi Z)),
\end{align*}
where $\alpha$ is a smooth function. It implies that
\begin{equation}\label{eq:Calphacondition}
s-s^{\ast}=4n^2\alpha\quad \textrm{and}\quad {\rm Ric}(\zeta,\zeta)=2n\alpha.
\end{equation}

\begin{cor}\label{cor:curvatureconditionscontact}
In the intrinsic torsion of an almost contact metric structure belongs to $\mathcal{D}_2$ class and satisfies Janssen--Vanhecke $C(\alpha)$ condition, then
\begin{equation*}
\alpha=\frac{1}{n(n-1)}\left(\frac{1}{2}s^{\rm alt}_{\alg{u}(n)^{\bot}}-{\rm div}(\nabla_{\zeta}\zeta)\right).
\end{equation*}
\end{cor}
\begin{proof}
Follows immediately by Proposition \ref{prop:contact1}(1).
\end{proof}

\begin{rem}\label{rem:finalremcontact}
In an analogous way as for $U(n)$--structures, it can be shown, with a little bit more effort, that the integral formula \eqref{eq:intformain} in this case is equivalent with the integral formula obtained in \cite{GMC2}.
\end{rem}

\section{Examples}

In thus section we apply obtained results to certain almost Hermitian and almost contact metric structures. First examples are simple illustration of obtained results for almost contact metric structures. We deal with this structure only since it has not been investigated from this point of view elsewhere. Almost Hermitian case is, due to well known facts contained in Proposition \ref{prop:curvaturerelations}, well understood from this perspective. In the end, we focus on more involving examples concerning homogeneous spaces, where we treat both cases -- almost Hermitian and almost contact metric structures. 

\begin{exa}
Let $(M,g,\varphi,\eta,\zeta)$ be an almost contact metric structure which is Sasaki, i.e., $\varphi$ satisfies the relation
\begin{equation*}
(\nabla_X\varphi)Y=g(X,Y)\xi-\eta(Y)X,\quad \nabla_X\zeta=-\varphi(X).
\end{equation*}
Thus, by \eqref{eq:inttorcontact}, the intrinsic torsion equals
\begin{equation*}
\xi_XY=g(X,\varphi(Y))\zeta,
\end{equation*}
or, in matrix notation
\begin{equation*}
\xi_X=\left( \begin{array}{cc} 0 & \varphi(X) \\ -\varphi(X)^{\top} & 0 \end{array} \right)\in\alg{u}(n)^{\bot}(TM)\quad\textrm{for $X\in{\rm ker}\eta$}
\end{equation*}
and $\xi_{\zeta}=0$. Hence we arrive in the $\mathcal{C}_5$ class. From the proof of Lemma \ref{lem:scalarcurvaturecompcontact} or, directly, by Proposition \ref{prop:contactinvariants} we see that $s_{\alg{u}(n)^{\bot}}^{\rm alt}$ does not vanish. In fact, by \cite[Table 1]{CG} we have
\begin{equation*}
i_6=-i_8=-i_{12}=\frac{1}{2n}i_{14}.
\end{equation*}
Thus, by the definition of $i_6$ and Lemma \ref{lem:scalarcurvaturecompcontact}
\begin{equation*}
s_{\alg{u}(n)^{\bot}}^{\rm alt}=(n-1)i_6=2n(n-1).
\end{equation*}
It can be shown \cite{JV} that the curvature tensor of a Sasakian structure satisfies $C(1)$ condition. Thus by \eqref{eq:Calphacondition}, $s-s^{\ast}=4n^2$ and ${\rm Ric}(\zeta,\zeta)=2n$. Since $\nabla_{\zeta}\zeta=-\varphi(\zeta)=0$, it follows that the left hand side of the second divergence formula in Proposition \ref{prop:contact1}(1) vanishes, whereas the right hand side equals
\begin{equation*}
\frac{1}{2}s^{\rm alt}_{\alg{u}(n)^{\bot}}-\frac{1}{4}(s-s^{\ast})+\frac{1}{2}{\rm Ric}(\zeta,\zeta)=n(n-1)-n^2+n=0.
\end{equation*} 
\end{exa}

\begin{exa}
Let $(M,g,\varphi,\eta,\zeta)$ be an almost contact metric structure which is Kenmotsu, i.e., $\varphi$ satisfies the following condition
\begin{equation*}
(\nabla_X\varphi)Y=g(\varphi X,Y)\zeta-\eta(Y)\varphi(X).
\end{equation*}
Hence, the intrinsic torsion, by the formula \eqref{eq:inttorcontact}, equals
\begin{equation*}
\xi_XY=g(X,Y)\zeta,\quad \xi_X \zeta=-X,\quad \xi_{\zeta}Y=0,\quad\xi_{\zeta}\zeta=0,
\end{equation*}
where $X,Y\in{\rm ker}\eta$. In matrix notation, for $X\in{\rm ker}\eta$,
\begin{equation*}
\xi_X=\left( \begin{array}{cc} 0 & X \\ -X^{\top} & 0 \end{array} \right)\in\alg{u}(n)^{\bot}(TM).
\end{equation*}
Hence, arguing as in the example above, $s^{\rm alt}_{\alg{u}(n)^{\bot}}$ does not vanish. Since, in this case,
\begin{equation*}
i_6=i_8=i_{12}=\frac{1}{2n}i_{10},
\end{equation*}
by Lemma \ref{lem:scalarcurvaturecompcontact},
\begin{equation*}
s^{\rm alt}_{\alg{u}(n)^{\bot}}=(1-n)i_6=2n(1-n).
\end{equation*}
It can be shown that Kenmotsu structure satisfies $C(-1)$ condition \cite{JV}. By \eqref{eq:Calphacondition}, $s-s^{\ast}=-4n^2$ and ${\rm Ric}(\zeta,\zeta)=-2n$. Therefore
\begin{equation*}
\frac{1}{2}s^{\rm alt}_{\alg{u}(n)^{\bot}}-\frac{1}{4}(s-s^{\ast})+\frac{1}{2}{\rm Ric}(\zeta,\zeta)=n(1-n)+n^2-n=0.
\end{equation*} 
Since for a Kenmotsu manifold $\nabla_X\zeta=X-\eta(X)\zeta$ it follows that
\begin{equation*}
{\rm div}\zeta=2n\quad\textrm{and}\quad \nabla_{\zeta}\zeta=0.
\end{equation*}
It follows that the second divergence formula in Proposition \ref{prop:contact1} is justified. 

Let us now justify condition (2) in Proposition \ref{prop:contact1}. It is known that the distribution ${\rm ker}\eta$ is integrable and umbilical \cite{KP}. Hence $T^{\eta}=0$ and $B^{\eta}=-\frac{1}{2n}g\otimes({\rm div}\zeta)\zeta$, which implies
\begin{equation*}
|T^{\eta}|^2-|B^{\eta}|^2=-2n
\end{equation*}
and
\begin{equation*}
\frac{1}{4}(s-s^{\ast})+\frac{1}{2}{\rm Ric}(\zeta,\zeta)-\frac{1}{2}s^{\rm alt}_{\alg{u}(n)^{\bot}}=-n^2-n-n(1-n)=-2n.
\end{equation*}
\end{exa}

\begin{exa}
Consider a generalized Heisenberg group $H(1,n)$ \cite{GMC2}. This is a Lie group consisting of square $n+2$ by $n+2$ martices of the form
\begin{equation*}
g=\left(\begin{array}{ccc}
I_n  & A^t & B^t \\ 0 & 1 & c \\ 0 & 0 & 1
\end{array}\right),
\end{equation*}
where $A=(a_1,\ldots,a_n)$ and $B=(b_1,\ldots,b_n)$ are elements of $\mathbb{R}^n$ and $c\in\mathbb{R}$. Then $H(1,n)$ is nilpotent of dimension $2n+1$. $H(1,n)$ has a global coordinate system $(x^i,x^{n+1},z)$, $i=1,2,\ldots,n$ given by
\begin{equation}\label{eq:exacontactheisenbergbasis}
x^i(g)=a_i,\quad x^{n+i}(g)=b_i,\quad z(g)=c.
\end{equation}
We choose a Riemannian metric such that the left--invariant basis
\begin{equation*}
X_i=\frac{\partial}{\partial x^i},\quad Y_i=\frac{\partial}{\partial x^{n+i}},\quad Z=\frac{\partial}{\partial z}+\sum_i x^i\frac{\partial}{\partial x^{n+i}}
\end{equation*}
is an orthonormal one. Nonzero components of the Levi--Civita connection are given by \cite{GMC2}
\begin{align*}
& \nabla_{X_i}X_{n+i}=\nabla_{X_{n+i}}X_i=-\frac{1}{2}Z,\\
& \nabla_{X_i}Z=-\nabla_ZX_i=\frac{1}{2}X_{n+1},\\
& \nabla_{X_{n+i}}Z=\nabla_ZX_{n+i}=\frac{1}{2}X_i.
\end{align*}
Thus, the integrability tensor $T^{\eta}$ vanishes and nonzero components of the the second fundamental form are equal  
\begin{equation*}
B^{\eta}(X_i,X_{n+i})=B^{\eta}(X_{n+i},X_i)=-\frac{1}{2}Z.
\end{equation*}
Hence, $|B^{\eta}|^2=\frac{n}{2}$. Moreover, by the formula for the curvature tensor \cite{GMC2}
\begin{align*}
& R(X_i,X_j,X_{n+i},X_{n+j})=\frac{1}{4}\quad i\neq j, && R(X_i,X_{n+j},X_j,X_{n+i})=-\frac{1}{4},\\
& R(X_i,Z,X_i,Z)=\frac{3}{4}, && R(X_{n+i},Z,X_{n+i},Z)=-\frac{1}{4}
\end{align*}
with remaining components vanishing, we see that
\begin{equation*}
s=-\frac{n}{2},\quad {\rm Ric}(Z,Z)=-\frac{n}{2}.
\end{equation*}

Consider on $H(1,n)$ an almost contact structure induced by the Reeb field $\zeta=Z$ and a dual one form $\eta=dz$ with a compatible endomorphism $\varphi$ such that $\varphi(X_i)=X_{n+i}$, $\varphi(X_{n+1})=-X_i$. Then we get a structure in a $\mathcal{C}_9$ class \cite{GMC2}. We have
\begin{equation*}
s^{\ast}=\frac{n}{2}.
\end{equation*}
It is easy to see that ${\rm div}Z=0$, hence the characteristic vector field vanishes (as noticed in Proposition \ref{prop:contactinvariants}(1)). We are ready to compute both sides of the divergence formula in Proposition \ref{prop:contact1}(2). The left hand side, clearly, equals $-\frac{n}{2}$. Since by Proposition \ref{prop:contactinvariants}(3), $s^{\rm alt}_{\alg{u}(n)^{\bot}}=0$, the right hand side is equal to
\begin{equation*}
\frac{1}{4}(s-s^{\ast})+\frac{1}{2}{\rm Ric}(\zeta,\zeta)=-\frac{n}{4}-\frac{n}{4}=-\frac{n}{2}.
\end{equation*}
\end{exa}

\subsection{Examples on reductive homogeneous spaces}
We will show that for a certain choice of $G$--structures on reductive homogeneous spaces induced from one parameter deformations of invariant Riemannian metrics, the characteristic vector field $\chi$ vanishes, hence, the divergence formula becomes point--wise formula. We justify this stating appropriate examples. We closely follow \cite[p. 140]{BFGK} and \cite{IA0}.

Let $K$ be a connected, compact Lie group and $H$ its closed, connected Lie subgroup. The quotient $K/H$ is a homogeneous space denoted by $M$. Assume additionally, that on the level of Lie algebras $\alg{k}=\alg{h}\oplus\alg{m}$, where $\alg{m}$ is the orthogonal complement with respect to some ${\rm ad}(H)$--invariant bilinear form ${\bf B}$  on $\alg{k}$. Deform ${\bf B}$ in the following way: assume $\alg{m}=\alg{m}_0\oplus\alg{m}_1$, where
\begin{equation}\label{eq:homogrel}
\begin{split}
&[\alg{h},\alg{m}_0]=\alg{m}_0,\quad [\alg{m}_0,\alg{m}_0]\subset\alg{h}\oplus\alg{m}_1,\\
&[\alg{h},\alg{m}_1]\subset\alg{m}_1,\quad [\alg{m}_1,\alg{m}_1]\subset\alg{h},\quad [\alg{m}_0,\alg{m}_1]\subset\alg{m}_0.
\end{split}
\end{equation}
For any $t>0$ we put
\begin{equation*}
{\bf B}_t={\bf B}|_{\alg{m}_0\times\alg{m}_0}+2t{\bf B}|_{\alg{m}_1\times\alg{m}_1}.
\end{equation*}
Form ${\bf B}_t$ defines an invariant riemannian metric $g_t$ on $M$. We will often write $g$ instead of $g_t$, if there is no confusion. The Levi--Civita connection of $g$ may be described as a linear map $\Lambda:\alg{m}\to\alg{so}(\alg{m})$ defined as follows \cite{BFGK}
\begin{align*}
\Lambda(X)Y &=\frac{1}{2}[X,Y]_{\alg{m}_1},\\
\Lambda(X)B &=t[X,B],\\
\Lambda(A)Y &=(1-t)[A,Y],\\
\Lambda(A)B &=0,
\end{align*}
where $X,Y\in\alg{m}_0$, $A,B\in\alg{m}_1$. 

Now, we consider a $G$--structure on $M$. Thus we have a decomposition $\alg{so}(\alg{m})=\alg{g}\oplus\alg{g}^{\bot}$, where we take orthogonal complement with respect to the Killing form on $\alg{so}(\alg{m})$. Then $\Lambda$ splits as $\Lambda=\Lambda_{\alg{g}}+\Lambda_{\alg{g}^{\bot}}$. $\Lambda_{\alg{g}}$ defines a $G$--connection $\nabla^G$, whereas $\Lambda_{\alg{g}^{\bot}}$ corresponds to the intrinsic torsion $\xi$.  

In the following two examples we introduce a $G$--structure via the same procedure. Denote by $x_0$ the coset $eH$, and let ${\rm ad}:H\to{\rm SO}(\alg{m})$ be the isotropy representation. Let $\varphi:\alg{m}\to\alg{m}$ be a linear map, which intertwines the isotropy representation. Since all tensor bundles on $M$ are associated bundles to the bundle $G\mapsto M$ with respect to the isotropy representation, it follows that $\varphi$ induces $(1,1)$--tensor field, in our case, almost hermitial or almost contact metric structure.

\begin{exa}
We follow \cite[p. 142]{BFGK}. Consider a complex flag manifold $F_{1,2}$ consisting of pairs $(l,V)$, where $l$ is $1$--dimensional complex subspace and $V$ is a complex $2$--dimensional subspace containing $l$ in $\mathbb{C}^3$. $U(3)$ acts transitively with a isotropy subgroup $H=U(1)\times U(1)\times U(1)$. Thus $F_{1,2}$ is a homogeneous space. On the level of Lie algebras $\alg{u}(n)=\alg{h}\oplus\alg{m}$, where $\alg{h}$ consists of diagonal matrices, whereas is a subspace of the form
\begin{equation*}
\alg{m}=\left\{\left(\begin{array}{ccc} 0 & a & b \\ -\bar{a} & 0 & c \\ -\bar{b} & -\bar{c} \end{array}\right),\quad a,b,c\in\mathbb{C}\right\}. 
\end{equation*}
$\alg{m}$ splits into two subspaces $\alg{m}_0$ and $\alg{m}_1$ given, respectively, by relations $c=0$ and $a=b=0$. Denote by ${\bf B}$ the Killing form on $\alg{u}(3)$, ${\bf B}(X,Y)=\frac{1}{2}{\rm Re}({\rm tr}XY))$. The inner product on $\alg{m}$ given by
\begin{equation*}
-{\bf B}|_{\alg{m}_0\times\alg{m}_0}+2t(-{\bf B})|_{\alg{m}_1\times\alg{m}_1}
\end{equation*}
defines a one parameter family of Riemannian metrics on $F_{1,2}$. The following basis is orthonormal with respect to the given inner product on $\alg{m}$:
\begin{equation*}
E_1=e_{12},\quad E_2=s_{12},\quad E_3=e_{13},\quad E_4=s_{13},\quad E_5=\frac{1}{\sqrt{2t}}e_{23},\quad E_6=\frac{1}{\sqrt{2t}}s_{23},
\end{equation*} 
where $e_{jk}$ is a skew--symmetric matrix with the $(j,k)$ entry equal to $1$ and $s_{jk}$ is a symmetric matrix with the $(j,k)$--entry equal to $i$ (and remaining elements except for $(k,j)$--entry equal to zero). We see that $\alg{m}_0$ is spanned by $E_1,E_2,E_3,E_4$, whereas $\alg{m}_1$ by $E_5,E_6$. It can be verified that the relations \eqref{eq:homogrel} hold.

Let us define an almost Hermitian structure on $F_{1,2}$. The isotropy representation ${\rm Ad}:H\to {\rm SO}(\alg{m})={\rm SO}(6)$ equals
\begin{equation*}
{\rm Ad}(t,r,s)=\left(\begin{array}{ccc} R_{t-s} & 0 & 0 \\ 0 & R_{t-r} & 0 \\ 0 & 0 & R_{s-r}\end{array}\right),
\end{equation*}
where $(t,r,s)$ denotes an element ${\rm diag}(e^{it},e^{is},e^{ir})\in H$ and $R_{\theta}$ is a rotation in $\mathbb{R}^2$ through an angle $\theta$. In order to define an almost Hermitian structure it suffices to define isotropy invariant $(1,1)$--tensor $J_0$ in $\alg{m}$ with $J_0^2=-1$. Let
\begin{equation*}
 J_0(E_1)=-E_2,\quad J_0(E_3)=E_4,\quad J_0(E_5)=-E_6.
\end{equation*}
The Levi--Civita connection of this almost Hermitian structure can be described by a map $\Lambda:\alg{m}\to\alg{so}(\alg{m})$,
\begin{align*}
&\Lambda(E_1)=\frac{\sqrt{t}}{\sqrt{2}}(e_{35}+e_{46}), &&\Lambda(E_2)=\frac{\sqrt{t}}{\sqrt{2}}(e_{45}-e_{36}),\\
&\Lambda(E_3)=\frac{\sqrt{t}}{\sqrt{2}}(e_{26}-e_{15}), &&\Lambda(E_4)=-\frac{\sqrt{t}}{\sqrt{2}}(e_{16}+e_{25}),\\
&\Lambda(E_5)=\frac{1-t}{\sqrt{2t}}(e_{13}+e_{24}), &&\Lambda(E_6)=\frac{1-t}{\sqrt{2t}}(e_{14}-e_{23}).
\end{align*}
The curvature tensor $R$ is given by
\begin{equation*}
R(X,Y)=[\Lambda(X),\Lambda(Y)]-\Lambda([X,Y]_{\alg{m}})-{\rm Ad}([X,Y]_{\alg{h}}),
\end{equation*}
where ${\rm Ad}:\alg{h}\to\alg{so}(\alg{m})$ denotes the differential of the isotropy representation,
\begin{equation*}
{\rm Ad}(H_1)=-e_{12}-e_{34},\quad {\rm Ad}(H_2)=e_{12}-e_{56},\quad {\rm Ad}(H_3)=e_{34}+e_{56}.
\end{equation*}
Here $H_k$ denotes the matrix $\frac{1}{2}s_{kk}$. 

Now we are ready to compute $s_{\alg{u}(3)^{\bot}}$. Firstly, notice that the commutators in $\alg{m}$ and its components in $\alg{h}$ are
\begin{align*}
&[E_1,E_2]=2H_1-2H_2, &&[E_1,E_3]=-\sqrt{2t}E_5, &&[E_1,E_4]=-\sqrt{2t}E_6,\\
&[E_1,E_5]=\frac{1}{\sqrt{2t}}E_3, &&[E_1,E_6]=\frac{1}{\sqrt{2t}}E_4, &&[E_2,E_3]=\sqrt{2t}E_6,\\
&[E_2,E_4]=-\sqrt{2t}E_5, &&[E_2,E_5]=\frac{1}{\sqrt{2t}}E_4, &&[E_2,E_6]=-\frac{1}{\sqrt{2t}}E_3,\\
&[E_3,E_4]=2H_1-2H_3, &&[E_3,E_5]=-\frac{1}{\sqrt{2t}}E_1, &&[E_3,E_6]=\frac{1}{\sqrt{2t}}E_2,\\
&[E_4,E_5]=-\frac{1}{\sqrt{2t}}E_2, &&[E_4,E_6]=-\frac{1}{\sqrt{2t}}E_1, &&[E_5,E_6]=\frac{1}{t}H_2-\frac{1}{t}H_3.
\end{align*}
Thus
\begin{align*}
&R(E_1,E_2)=(2-t)e_{34}+(t-2)e_{56}+4e_{12}, &&R(E_1,E_3)=(1-\frac{3}{2}t)e_{13}+(1-\frac{t}{2})e_{24},\\
&R(E_1,E_4)=(1-\frac{3}{2}t)e_{14}+(\frac{t}{2}-1)e_{23}, &&R(E_1,E_5)=(1-\frac{t}{2})e_{15}-\frac{t}{2}e_{26},\\
&R(E_1,E_6)=\frac{t}{2}e_{16}+(1-\frac{t}{2})e_{25}, &&R(E_2,E_3)=(\frac{t}{2}-1)e_{14}+(1-\frac{3}{2}t)e_{23},\\
&R(E_2,E_4)=(1-\frac{t}{2})e_{13}+(1-\frac{3}{2}t)e_{24}, &&R(E_2,E_5)=(1-\frac{t}{2})e_{16}+\frac{t}{2}e_{15},\\
&R(E_2,E_6)=(\frac{t}{2}-1)e_{15}+\frac{t}{2}e_{26}, &&R(E_3,E_4)=(2-t)e_{12}+4e_{34}-(2-t)e_{56},\\
&R(E_3,E_5)=\frac{t}{2}e_{35}+(1-\frac{t}{2})e_{46}, &&R(E_3,E_6)=\frac{t}{2}e_{36}+(\frac{t}{2}-1)e_{45},\\
&R(E_4,E_5)=(\frac{t}{2}-1)e_{36}+\frac{t}{2}e_{45}, &&R(E_4,E_6)=(1-\frac{t}{2})e_{35}+\frac{t}{2}e_{46},\\
&R(E_5,E_6)=(t-2)e_{12}+(2-t)e_{34}+\frac{2}{t}e_{56}.
\end{align*}
Moreover, notice, that $e_{12},e_{34},e_{56}\in \alg{u}(3)$. For the remaining elements, its projection to $\alg{u}(3)^{\bot}$ equals, respectively,
\begin{align*}
&e_{13}\mapsto\frac{1}{2}(e_{13}+e_{24}), &&e_{14}\mapsto\frac{1}{2}(e_{14}-e_{23}), &&e_{15}\mapsto\frac{1}{2}(e_{15}-e_{26}), \\
&e_{16}\mapsto\frac{1}{2}(e_{16}+e_{25}), &&e_{23}\mapsto\frac{1}{2}(e_{23}-e_{14}), &&e_{24}\mapsto\frac{1}{2}(e_{24}+e_{13}),\\
&e_{25}\mapsto\frac{1}{2}(e_{25}+e_{16}), &&e_{26}\mapsto\frac{1}{2}(e_{26}-e_{15}), &&e_{35}\mapsto\frac{1}{2}(e_{35}+e_{46}),\\
&e_{36}\mapsto\frac{1}{2}(e_{36}-e_{45}), &&e_{45}\mapsto\frac{1}{2}(e_{45}-e_{36}), &&e_{46}\mapsto\frac{1}{2}(e_{46}+e_{35}).
\end{align*} 
Hence
\begin{equation*}
s_{\alg{u}(3)^{\bot}}=8(2-t).
\end{equation*}
We could obtain above relation by applying the formula \eqref{eq:sunbot}. It is easy to see that $s=2(-13+3t-\frac{2}{t})$ and $s^{\ast}=2(3-5t-\frac{2}{t})$.

Let us turn to computations of the intrinsic torsion and its components. By above relations we see that $\Lambda:\alg{m}\to\alg{u}(3)^{\bot}$. Hence, the minimal connection $\nabla^{U(3)}$ is induced by a zero map and the intrinsic torsion corresponds to $\Lambda$. Since $\Lambda(E_i)E_i=0$ it follows that the characteristic vector field $\chi$ vanishes. Moreover,
\begin{equation*}
\sum_{i,j}\skal{\Lambda(E_i)E_j}{\Lambda(E_j)E_i)}=4(t-2)
\end{equation*}
corresponds to $|\xi^{\rm sym}|^2-|\xi^{\rm alt}|^2$, hence the main divergence formula, which reduces to $\frac{1}{2}s_{\alg{u}(3)^{\bot}}=|\xi^{\rm alt}|^2-|\xi^{\rm sym}|^2=4(2-t)$, is justified.

Let us look at the Gray--Hervela classes induced by $t$ for each choice of $t>0$. Since $\chi=0$, the considered almost Hermitian structure is of type $\mathcal{W}_1\oplus\mathcal{W}_2\oplus\mathcal{W}_3$. Simple calculations show that $\Lambda$, hence $\xi$, satisfies $\xi_{JX}JY=-\xi_XY$. Thus, by Proposition \ref{prop:GHclassesproperties}, the considered structures are of type $\mathcal{W}_1\oplus\mathcal{W}_2$. Moreover, it is nearly K\"ahler, i.e. in $\mathcal{W}_1$, if and only if $t=\frac{1}{2}$. By above considerations, we see that for $t<2$, $s_{\alg{u}(3)^{\bot}}>0$ and for $t>2$, $s_{\alg{u}(3)^{\bot}}<0$.
\end{exa}

\begin{exa}
We follow very closely the approach by Agricola \cite{IA0}. Consider the $5$--dimensional Steifel manifolds $V_{4,2}=SO(4)/SO(2)$. We embed $S(2)$ as a lower diagonal block. We have the splitting $\alg{so}(4)=\alg{so}(2)\oplus\alg{m}$ with respect to the Killing form ${\bf B}$, where
\begin{equation*}
\alg{m}=\left\{\left(\begin{array}{cc} A & X \\ -X^{\top} & 0 \end{array}\right)\mid
A=\left(\begin{array}{cc} 0 & -a \\ a & 0 \end{array}\right),\quad X\in\mathcal{M}_{2\times 2}(\mathbb{R})
\right\}.
\end{equation*} 

There is a one parameter family $g_t$ of Riemannian metrics on $V_{4,2}$ constructed by Jensen \cite{GJ}, which are obtained from the invariant dot product on $\alg{m}$
\begin{equation*}
\skal{(a,X)}{(b,Y)}={\bf B}(X,Y)+2t\,ab,
\end{equation*}
where $(a,X),(b,Y)$ denote the elements in $\alg{m}$ and ${\bf B}(X,Y)=\frac{1}{2}{\rm tr}(X^{\top}Y)$.
Denoting the canonical basis in $\alg{so}(4)$ by $(e_{ij})$, i.e., $e_{ij}$ is a skew symmetric matrix with the $(i,j)$--entry equal to $-1$ (be aware of the difference with the sign convention comparing to the previous example), we have an orthonormal basis 
\begin{equation*}
E_1=e_{13},\quad E_2=e_{14},\quad E_3=e_{23},\quad E_4=e_{24},\quad E_5=\frac{1}{\sqrt{2t}}e_{12}
\end{equation*}
in $\alg{m}$. We have the following Lie bracket relations
\begin{align*}
&[E_1,E_2]=e_{34}, &&[E_1,E_3]=\sqrt{2t}E_5, &&[E_1,E_4]=0,\\
&[E_1,E_5]=-\frac{1}{\sqrt{2t}}E_3, &&[E_2,E_3]=0, &&[E_2,E_4]=\sqrt{2t}E_5,\\
&[E_2,E_5]=-\frac{1}{\sqrt{2t}}E_4, &&[E_3,E_4]=e_{34}, &&[E_3,E_5]=\frac{1}{\sqrt{2t}}E_1,\\
&[E_4,E_5]=\frac{1}{\sqrt{2t}}E_2.
\end{align*}
Two brackets, $[E_1,E_2]$ and $[E_3,E_4]$, have $\alg{so}(2)$-components but no $\alg{m}$--components. Denoting by $\alg{m}_0$ the subspace spanned by $E_1,E_2,E_3,E_4$ and by $\alg{m}_1$ the one--dimensional space spanned by $E_5$, we see that relations \eqref{eq:homogrel} are satisfied. The Levi--Civita connection of this homogeneous space, is described by a map $\Lambda:\alg{m}\to\alg{so}(\alg{m})=\alg{so}(5)$ of the form \cite{GJ,IA0} 
\begin{align*}
&\Lambda(E_1)=\sqrt{\frac{t}{2}}e_{35}, &&\Lambda(E_2)=\sqrt{\frac{t}{2}}e_{45}, &&\Lambda(e_3)=-\sqrt{\frac{t}{2}}e_{15},\\
&\Lambda(E_4)=-\sqrt{\frac{t}{2}}e_{25}, &&\Lambda(E_5)=\frac{1-t}{\sqrt{2t}}(e_{13}+e_{24}).
\end{align*}
The curvature tensor is then given by
\begin{equation*}
R(X,Y)=[\Lambda(X),\Lambda(Y)]-\Lambda([X,Y]_{\alg{m}})-{\rm Ad}([X,Y]_{\alg{h}}),
\end{equation*}
where ${\rm Ad}$ is the differential of the isotropy representation ${\rm Ad}:SO(2)\to SO(\alg{m})$,
\begin{equation*}
{\rm Ad}(g)=\left(\begin{array}{ccc} g & 0 & 0 \\ 0 & g & 0 \\ 0 & 0 & 0\end{array}\right),\quad g\in SO(2).
\end{equation*}
Hence, $R$ has the following components
\begin{align*}
&R(E_1,E_2)=-e_{12}+(\frac{t}{2}-1)e_{34}, && R(E_1,E_3)=(\frac{3}{2}t-1)e_{13}+(t-1)e_{24},\\
&R(E_1,E_4)=\frac{t}{2}e_{23}, &&R(E_1,E_5)=-\frac{t}{2}e_{15},\\
&R(E_2,E_3)=\frac{t}{2}e_{14}, &&R(E_2,E_4)=(t-1)e_{13}+(\frac{3}{2}t-1)e_{24},\\
&R(E_2,E_5)=-\frac{t}{2}e_{25}, &&R(E_3,E_4)=(\frac{t}{2}-1)e_{12}-e_{34},\\
&R(E_3,E_5)=-\frac{t}{2}e_{35}, &&R(E_4,E_5)=-\frac{t}{2}e_{45}.
\end{align*}

Now we introduce an almost contact structure on $V_{4,2}$ by defining an isotropy invariant map $\varphi:\alg{m}\to\alg{m}$, which is, in terms of the basis $(E_i)$, given by a matrix
\begin{equation*}
\varphi=\left(\begin{array}{ccccc}
0 & 0 & 1 & 0 & 0 \\
0 & 0 & 0 & 1 & 0 \\
-1 & 0 & 0 & 0 & 0 \\
0 & -1 & 0 & 0 & 0 \\
0 & 0 & 0 & 0 & 0 
\end{array}\right).
\end{equation*}
One can check that $\varphi$, in deed, defines an almost contact structure with the Reeb field $\zeta$ induced by $E_5$ and compatible with the metric $g_t$. Notice, that the fundamental form $F(X,Y)=g_t(X,\varphi(Y))$ induced by $\varphi$ if $F=e_{13}+e_{24}$, which is proportional to $dE_5$ making the structure just Sasaki structure \cite{IA0}.

To derive the formula for $s_{\alg{u}(2)^{\bot}}$ we need to compute the $\alg{u}(2)^{\bot}$--component of $R$, i.e., project $R(E_i,E_j)$ to the second factor with respect to the decomposition $\alg{so}(\alg{m})=\alg{u}(2)\oplus\alg{u}(2)^{\bot}$. It is easy to see that $e_{i5}\in\alg{u}(2)^{\bot}$ and $e_{13},e_{24}\in\alg{u}(2)$. Moreover, $e_{13}$ and $e_{24}$ project to $v=\frac{1}{2}(e_{13}-e_{24})$ and $-v$, respectively. Finally, $e_{12}$ projects to $w=\frac{1}{2}(e_{12}-e_{34})$, whereas $e_{34}$ projects to $-w$. Thus
\begin{equation*}
s_{\alg{u}(2)^{\bot}}=6t.
\end{equation*}
We could compute $s_{\alg{u}(2)^{\bot}}$ using Lemma \ref{lem:scalarcurvaturecompcontact} and noticing that
\begin{equation*}
s=2(4-t),\quad s^{\ast}=2(4-5t),\quad {\rm Ric}(\zeta,\zeta)=2t.
\end{equation*}

Let us describe the intrinsic torsion. We see that $\Lambda_{\alg{u}(3)^{\bot}}(E_i)=\Lambda(E_i)$ for $i=1,2,3,4$, and $\Lambda_{\alg{u}(3)^{\bot}}(E_5)=0$. Hence, the characteristic vector field $\chi$, which corresponds to $\sum_i \Lambda_{\alg{u}(2)^{\bot}}(E_i)E_i$, vanishes and, it is not hard to see, that $|\xi^{\rm alt}|^2-|\xi^{\rm sym}|^2$, which corresponds to $-\sum_{i,j}\skal{\Lambda_{\alg{u}(2)^{\bot}}(E_i)E_j}{\Lambda_{\alg{u}(2)^{\bot}}(E_j)E_i}$, equals $2t$. Finally, directly from the definition it is not hard to see that $s^{\rm alt}_{\alg{u}(2)^{\bot}}=2t$. Thus
\begin{equation*}
\frac{1}{2}s^{\rm alt}_{\alg{u}(2)^{\bot}}-\frac{1}{2}s_{\alg{u}(2)^{\bot}}+|\xi^{\rm alt}|^2-|\xi^{\rm sym}|^2=t-3t+2t=0.
\end{equation*}
Notice, finally, that by the classification of all possible intrinsic torsion modules \cite{CG}, since $i_2=|\xi^{\rm sym}|^2-|\xi^{\rm alt}|=-t<0$, we arrive in a pure $\mathcal{C}_1$ class.
\end{exa}

\end{document}